\documentclass[12pt]{amsart}
 
\usepackage{amsmath,amssymb,amsthm,amsfonts,mathtools}
\usepackage{geometry}
\usepackage[colorlinks]{hyperref}
\usepackage{enumitem}
\usepackage{booktabs}
\usepackage{microtype}
\usepackage{tikz-cd}
\usepackage{bm}

\usepackage{verbatim}
\usepackage{url}

\newtheorem{theorem}{Theorem}[section]

\newtheorem{lemma}[theorem]{Lemma}

\newtheorem{problem}{Problem}

\newcommand{\bsigma}{\bm{\vec{\sigma}}}

\usepackage{caption}
\usepackage{array} % add to the preamble to make the table
\newcolumntype{M}[1]{>{\centering\arraybackslash}m{#1}} % add to the preamble to make the table
\renewcommand{\qedsymbol}{$\blacksquare$}
\title{Low Lying Zeros of Dedekind Zeta Functions}

\author{Timothy S. Trudgian}

\author{Sebastian Zuniga-Alterman}

\newcommand{\Addresses}{{
  {\footnotesize
\ \\

\noindent T.S. TRUDGIAN - \textsc{School of Science, The University of New South Wales, Canberra, Northcott Drive, Campbell ACT 2612, AUSTRALIA.}\\
  \texttt{timothy.trudgian@unsw.edu.au}\\
  
   \noindent S.~ZUNIGA-ALTERMAN - \textsc{School of Science, The University of New South Wales, Canberra, Northcott Drive, Campbell ACT 2612, AUSTRALIA.}\\
  \texttt{s.zuniga$\_$alterman@unsw.edu.au}

}}}

\thanks{Thanks are due to A. Odlyzko for insightful discussions about the topic of this article.}

\thanks{The authors are supported by the Australian Research Council Discovery Project DP240100186.}

\begin{document}

\begin{abstract}
Let $K$ be a number field of degree $n_K$, and let $h(K)$ denote the smallest absolute value of the imaginary part of a non-trivial zero of its Dedekind zeta function $\zeta_K$. In 1979, Hoffstein proved the unconditional bound $h(K)\leq 0.87+o(1)$ as $n_K\to\infty$. In a 1990 survey, Odlyzko announced the stronger bound $h(K)\leq 0.54+o(1)$ without proof. No improvement on Hoffstein's bound has since appeared with a published proof. 

We construct a comparison kernel from a linear combination of three identities and show that Hoffstein's original parameters give $h(K)<0.825$ for sufficiently large $n_K$. We generalise Hoffstein's method via Stechkin's trick and prove $h(K)<0.58$ for all sufficiently large $n_K$, giving the best proved estimate to date. Moreover, $\zeta_K$ has at least $c n_K$ non-trivial zeros with imaginary part in $[0,0.58]$, counted with multiplicity, for some $c>0$ and $n_K$ large enough.
\end{abstract}

\maketitle

%%%%%%%%%%%%%%%%%%%%%%%%%%%%%%%%%%%%%%%%%%%%%%%%%
%%%%%%%%%%%%%%%%%%%%%%%%%%%%%%%%%%%%%%%%%%%%%%%%%
\section{Introduction}\label{horse}
%%%%%%%%%%%%%%%%%%%%%%%%%%%%%%%%%%%%%%%%%%%%%%%%%
%%%%%%%%%%%%%%%%%%%%%%%%%%%%%%%%%%%%%%%%%%%%%%%%%

Let $K$ be a number field (a finite-degree field extension of $\mathbb{Q}$) of degree $n_K=[K:\mathbb{Q}]$, discriminant $d_K=\text{disc}(K/\mathbb{Q})$ and signature $(r_1,r_2)$. Thus $K$ has $r_1$
real embeddings and $2r_2$ complex embeddings. Let $\zeta_K$ denote its
Dedekind zeta function.

We define the height of the first non-trivial zero in the critical strip by
\begin{equation*}
    h(K):=\min\{|\gamma|\geq 0:\zeta_K(\beta+i\gamma)=0,0<\beta<1\}.
\end{equation*}
Note that a real non-trivial zero gives $h(K)=0$.
Hoffstein \cite{Hoffstein*79} observed that Odlyzko's methods for bounding
discriminants could also be applied to the problem of low-lying zeros of
$\zeta_K$. In particular, he proved that
\begin{equation}\label{stein}
        h(K)\leq  0.87+o(1),\qquad n_K\to\infty.
\end{equation} 
Later, Odlyzko \cite{Odlyzko*90} attributed to unpublished work
\cite{Odlyzko***} the  estimate
\begin{equation}\label{Odlyzko}
        h(K)\leq  0.54+o(1),\qquad n_K\to\infty.
\end{equation}
Neugebauer \cite{Neugebauer*88} proved that every $\zeta_K$ admits a \emph{non-real} zero of height less than $60$.

Estimates can be improved if one assumes the \emph{generalised Riemann hypothesis} (GRH), namely, that all the nontrivial zeros of $\zeta_K$ have real part equal to $1/2$. 
%and bounds for $h(K)$ under GRH are called \emph{conditional} bounds.
Let
\begin{equation*}
    \tau(K):=\min\{|\gamma|>0:\zeta_K(1/2+i\gamma)=0\}.
\end{equation*}
In other words, $\tau(K)$ is the first positive height of a zero on the
critical line. Omar \cite{Omar*00} proved, under GRH, an upper bound which for fixed degree reads as
    $
    \tau(K)\ll_{n_K} \frac{1}{\log\log |d_K|}$.
    Moreover, Kala \cite{Kala*25} recently gave an explicit version of Omar's argument, still assuming GRH.

Odlyzko's survey \cite{Odlyzko*90} on discriminants, regulators, class numbers and zeros records three consequences of the explicit formula, which are stated without proof. First, under GRH, $\zeta_K$ has a zero on the critical line at height $\ll \frac{1}{\log n_K}$ as $n_K\to\infty$. Second, unconditionally, $\zeta_K$ has a zero below $14$ for every $K$ with $n_K\geq 2$. Third is the estimate in \eqref{Odlyzko}.

%Those results are attributed to a manuscript that was not published, confirmed by Odlyzko himself \cite{Odlyzko***}. In this paper, we obtain the best up to date improvement on \eqref{Hoffstein:thm}. This is precisely Theorem \ref{main}.
We began this manuscript in an attempt to prove (and possibly to improve) the estimate in \eqref{Odlyzko}. We proved the following improvement on Hoffstein's bound \eqref{stein}.
%
%Since Odlyzko's decades-old claim \eqref{Odlyzko} was never published, any published improvement on Hoffstein's constant is of independent interest. We prove the following theorem, which is the best up to date result.
\begin{theorem}\label{main} For sufficiently large $n_K$, we have
$        h(K)< 0.58.$
\end{theorem}
Theorem \ref{main} is proved in \S\ref{proof-main} by generalising Hoffstein's method using
Stechkin's trick --- see \S \ref{bushel} for details. We consider the same comparison weight introduced by
Hoffstein to detect low-lying zeros, and deform it by a linear shift of the
parameters. Moreover, we have the following quantitative result, which we prove in \S\ref{proof:main:quantitative}.
%The relevant point is that the original and shifted prime ideal contributions remain non-negative, and therefore may be discarded by positivity when proving a lower bound.
Let
\begin{equation}\label{N_K}
N_K(T):=\#\{\rho=\beta+i\gamma:\zeta_K(\rho)=0,\ 0<\beta<1,\ 0\leq\gamma\leq  T\},
\end{equation} 
where zeros are counted with multiplicity.
\begin{theorem}\label{main:quantitative}
There exists an absolute constant $c>0$ such that
\begin{equation*}
N_K(0.58)\geq c\ n_K
\end{equation*}
for every number field $K$ of sufficiently large degree $n_K$. 
%In particular, the number of low-lying zeros of $\zeta_K$, counted with multiplicity, grows at least linearly with the degree $n_K$.
\end{theorem}

We remark that our method is considerably simpler, and perhaps weaker, than that used by Odlyzko. Therefore, one could expect that the introduction of Stechkin's trick, combined with Odlyzko's methods, could improve our Theorem \ref{main} further.

At the heart of our method is an optimisation problem (Problem \ref{Optimisation:A}). This problem is then reduced to simpler ones, which we eventually solve via numerical analysis implemented in SageMath 10.9 \cite{SageMath*26}. See diagram below.
\begin{equation*}
\begin{tikzcd}[
    column sep=1.7em,
    row sep=0.1em,
    cells={nodes={inner sep=1pt}}
]
        & \textnormal{Problem \ref{Optimisation:C}}
          \arrow[dr,Rightarrow,shorten <=3pt,shorten >=3pt] & \\
        & & \textnormal{Problem \ref{Optimisation:B}}
          \arrow[r,Rightarrow,shorten <=3pt,shorten >=3pt]
          & \textnormal{Problem \ref{Optimisation:A}} \\
        & \textnormal{Problem \ref{Optimisation:D}}
          \arrow[ur,Rightarrow,shorten <=3pt,shorten >=3pt] &
\end{tikzcd}
\end{equation*}
In other words, either Problem \ref{Optimisation:C} or \ref{Optimisation:D} is sufficient to prove Problem \ref{Optimisation:B}, which is sufficient to prove the main Problem \ref{Optimisation:A}.
After defining our notation below, we introduce the necessary properties of $\zeta_{K}$ in \S \ref{bushel}. We set up Hoffstein's argument, with our variant using Stechkin's trick, in \S \ref{peck}. We introduce symmetric weight functions in \S \ref{ton} to use in our optimisation problem in \S \ref{rod}. We record our results in \S \ref{outputs} and end with some concluding remarks in \S \ref{numerics}.

\subsection{Notation}

We write $\bsigma=(\sigma_0,\sigma_1,\sigma_2)$. The variable $s$ denotes a complex number, and we write $\Re s=\sigma$, $\Im s=t$.
If $\rho$ is a non-trivial zero of $\zeta_K$, we write $ \rho=\beta+i\gamma$, $0<\beta<1.$
All zero sums are taken over non-trivial zeros in the critical strip and are
counted with multiplicity. 
%Thus sums or products of the form
%\begin{equation*}
%        \sum_{\rho=\beta},\qquad \prod_{\rho=\beta}, \qquad\sum_{\substack{\rho=\beta+i\gamma\\ \gamma>0}},
 %       \qquad\prod_{\substack{\rho=\beta+i\gamma\\ \gamma>0}}
%\end{equation*}
%mean respectively a sum or product over real non-trivial zeros $\rho=\beta$, $0<\beta<1$, counted
%with multiplicity, and a sum over one representative from each non-real conjugate pair, counted
%with multiplicity. 

{\em $O^*$ notation}: let $f$ be a complex-valued function and $g$ a non-negative real-valued function, both defined on a domain $D$.
We write $f(x)=O^*(g(x))$ on $D$ if $|f(x)|\leq g(x)$ for $x\in D$.

%%%%%%%%%%%%%%%%%%%%%%%%%%%%%%%%%%%%%%%%%%%%%%%%%
%%%%%%%%%%%%%%%%%%%%%%%%%%%%%%%%%%%%%%%%%%%%%%%%%
\section{The Dedekind zeta function}\label{bushel}
%%%%%%%%%%%%%%%%%%%%%%%%%%%%%%%%%%%%%%%%%%%%%%%%%
%%%%%%%%%%%%%%%%%%%%%%%%%%%%%%%%%%%%%%%%%%%%%%%%%

The Dedekind zeta function associated with the number field $K$ is defined for $\Re s>1$ by
$ \zeta_K(s) =\sum_{\mathfrak a\in I_K^*}(N\mathfrak a)^{-s}$,
where $I_K^*$ is the set of nonzero integral ideals of $K$ and $N\mathfrak a=[\mathcal O_K:\mathfrak a]$
is the ideal norm. 
Furthermore, the ideal norm function $N:I_K^*\to\mathbb{N}$ is multiplicative, and thus allows us to express $\zeta_K$ by means of an Euler product in the region $\{\Re s=\sigma>1\}$. This is based on unique factorisation of nonzero integral ideals in $I_K^*$ into nonzero prime ideals belonging to $P_K^*$. Thus we have the following identity
\begin{equation}\label{Euler}
\zeta_K(s)=\prod_{\mathfrak p\in P_K^*}\bigg(1-\frac{1}{(N\mathfrak{p})^s}\bigg)^{-1},\qquad\qquad\sigma>1.
\end{equation}

\begin{lemma}\label{ideals} 
For $\sigma>1$, both 
$-\frac{\zeta_K'}{\zeta_K}(\sigma)$ and $\big(\frac{\zeta_K'}{\zeta_K}\big)'(\sigma)$ are positive.
Also, for $\sigma_0,\sigma_1>1$,
\begin{equation*}
\log\frac{\zeta_K(\sigma_1)}{\zeta_K(\sigma_0)}=\sum_{\mathfrak p\in P_K^*}\log \bigg( \frac{1-(N\mathfrak p)^{-\sigma_0}}{1-(N\mathfrak p)^{-\sigma_1}}\bigg).
\end{equation*}
\end{lemma}
\begin{proof} The logarithmic identity follows directly from \eqref{Euler}. On the other hand, again by the Euler product expression \eqref{Euler} for $\zeta_K$, since $\Re s>1$, we have
\begin{equation*}
 -\frac{\zeta_K'}{\zeta_K}(s)=\sum_{\mathfrak p\in P_K^*}\sum_{m\geq 1}
\frac{\log N\mathfrak p}{(N\mathfrak p)^{ms}},\qquad 
\bigg(\frac{\zeta_K'}{\zeta_K}\bigg)'(s)=\sum_{\mathfrak p\in P_K^*}\sum_{m\geq 1}
\frac{m\log^2 N\mathfrak p}{(N\mathfrak p)^{ms}}.
\end{equation*}
By selecting $s=\sigma>1$, we obtain that the above expressions are positive.
\end{proof}

\subsection{Functional equation} Let
\begin{equation}\label{Lambda}
\Lambda(s)=|d_K|^{s/2}\bigg(\frac{\Gamma(s/2)}{\pi^{s/2}}\bigg)^{r_1}\bigg(\frac{\Gamma(s)}{(2\pi)^s}\bigg)^{r_2}\zeta_K(s)
\end{equation}
The function $\Lambda$ is invariant by the change $s\mapsto 1-s$, whence the functional equation
\begin{equation}\label{functional}
\Lambda(s)=\Lambda(1-s).
\end{equation}
Further, the function $s(s-1)\Lambda(s)$ is entire of order $1$ whose zeros are the nontrivial zeros of $\zeta_K$. Thus, the Dedekind zeta function has an analytic extension to $\mathbb{C}$ except for a simple pole at $s=1$ and its non-trivial zeros are those belonging to the critical strip $\{0<\Re s<1\}$. 

\subsection{Zeros of Dedekind zeta functions} Let $\rho=\beta+i\gamma$ be a non-trivial zero of $\zeta_K$. Thanks to the functional equation \eqref{functional}, we have that $1-\rho$ is also a zero of $\zeta_K$. Moreover, since $\overline{\zeta_K(s)}=\zeta_K(\overline{s})$,
the zeros of $\zeta_K$ appear in quadruples $\{\rho,\overline{\rho},1-\rho,1-\overline{\rho}\}$ with possible coincidences.
%
%On the other hand, the generalised Riemann hypothesis (GRH) states that all the nontrivial zeros of $\zeta_K$ lie on the vertical line $\{\sigma=\tfrac{1}{2}\}$, that is, we always have $\beta=\tfrac{1}{2}$. 
%
Note that those zeros are always counted with multiplicities.

By recalling definition \eqref{Lambda}, we have the following Hadamard product
\begin{equation}\label{Hadamard}
\xi_K(s):=s(s-1)\Lambda(s)
        =e^{a+bs}\prod_{\rho}\bigg(1-\frac{s}{\rho}\bigg)e^{s/\rho},
\end{equation}
where the product runs over the non-trivial zeros $\rho$ of
$\zeta_K$, and converges normally on $\mathbb{C}$. Note that $\xi_K(1)=\xi_K(0)>0$, thanks to definition \eqref{Lambda} and $\text{Res}_{s=1}\zeta_K(s)>0$.

\begin{lemma}\label{Hadamard:grouped} 
There exists a constant $C>0$ such that
\begin{equation}\label{pound}
\xi_K(s)=
C\prod_{\rho=\beta}\bigg(1-\frac{s}{\rho}\bigg)
\prod_{\substack{\rho=\beta+i\gamma\\ \gamma>0}}
\bigg(1-\frac{s}{\rho}\bigg) \bigg(1-\frac{s}{\overline{\rho}}\bigg).
\end{equation}
%where the products are taken with multiplicity.
Consequently, for $s$ not a non-trivial zero of $\zeta_K$,
\begin{equation}\label{stone}
\frac{\xi_K'}{\xi_K}(s) =
\sum_{\rho=\beta} \frac{1}{s-\rho}+\sum_{\substack{\rho=\beta+i\gamma\\ \gamma>0}}\bigg(\frac{1}{s-\rho}+ \frac{1}{s-\overline{\rho}}\bigg).
\end{equation}
\end{lemma}

\begin{proof}
In the Hadamard product \eqref{Hadamard} we separate the real zeros from the non-real zeros and group the latter in
conjugate pairs; this reordering is justified by absolute convergence. There are only finitely many real non-trivial zeros and for $\rho=\beta+i\gamma$, $\gamma>0$, we have
\begin{equation*}
e^{s/\rho}e^{s/\overline{\rho}}
=e^{s\frac{2\beta}{\beta^2+\gamma^2}}=e^{O_s\big(\frac{1}{\gamma^2}\big)}.
\end{equation*}
Hence, the corresponding sum over zeros $\rho=\beta+i\gamma$, $\gamma>0$, converges thanks to the zero counting estimate $N_K(T)\ll_K T(\log T+2)$, $T\geq 1$, where zeros are counted with multiplicity and $N_K(T)$ is defined in \eqref{N_K}.
Subsequently, all exponential correction factors combine with
$e^{bs}$ into a single exponential factor $e^{Bs}$.  Thus we may write
\begin{equation}\label{LD}
\xi_K(s) =e^{a+Bs}\prod_{\rho=\beta} \bigg(1-\frac{s}{\rho}\bigg)\prod_{\substack{\rho=\beta+i\gamma\\ \gamma>0}} \bigg(1-\frac{s}{\rho}\bigg)\bigg(1-\frac{s}{\overline{\rho}}\bigg)=e^{a+Bs}P(s).
\end{equation}

On the other hand, the functional equation
$\xi_K(s)=\xi_K(1-s)$ implies that the non-trivial zeros are invariant under
$\rho\longmapsto1-\rho$. Moreover, for a non-real zero
$\rho=\beta+i\gamma$ with $\gamma>0$,
\begin{equation*}
\bigg(1-\frac{1-s}{\rho}\bigg)\bigg(1-\frac{1-s}{\overline{\rho}}\bigg)=
\frac{|1-\rho|^2}{|\rho|^2}\bigg(1-\frac{s}{1-\rho}\bigg)\bigg(1-\frac{s}{1-\overline{\rho}}\bigg).
\end{equation*}
Now, the map $\rho\longmapsto1-\overline{\rho}$ permutes the zeros in the upper half-plane; therefore, rearranging the absolutely convergent product gives
\begin{equation*}
P(1-s)=P(s)\prod_{\rho=\beta}\frac{\rho-1}{\rho}
\prod_{\substack{\rho=\beta+i\gamma\\ \gamma>0}}
\frac{|1-\rho|^2}{|\rho|^2}=P(s)\prod_{\rho=\beta}\frac{\rho-1}{\rho}
\prod_{\substack{\rho=\beta+i\gamma\\ \gamma>0}}\bigg(1+O^{*}\bigg(\frac{1}{\gamma^2}\bigg)\bigg)=cP(s),
\end{equation*}
where $c\in\mathbb{C}\setminus{\{0\}}$, since there are finitely many real zeros $\rho=\beta$ and the above infinite product converges absolutely to a non-zero constant. Hence
\begin{equation*}e^{a+Bs}P(s)=\xi_K(s)=\xi_K(1-s)=e^{a+B(1-s)}P(1-s)=ce^{a+B(1-s)}P(s).
\end{equation*}
Since $P$ is not
identically zero, this forces $e^{B(2s-1)}=c$
for all $s\in\mathbb{C}$, and therefore $B=0$ and $c=1$. Consequently, $\xi_K(s)=e^aP(s)$ and since $P(0)=1$, we conclude that $\xi_K(1)=e^a=C>0$ and we obtain \eqref{pound}.

On the other hand, we obtain \eqref{stone} by taking the logarithmic derivative in \eqref{LD} for $s$ not at non-trivial zero of $\zeta_K$ and observing that $
\frac{\xi_K'}{\xi_K}(s) =\frac{P'(s)}{P(s)}. $
\end{proof}

The Landau--Stark formulas correspond to the logarithmic derivative of \eqref{Hadamard}.

\begin{theorem}[Landau--Stark identity]\label{LSI}  Let $\psi(s)=\frac{\Gamma'}{\Gamma}(s)$ be the logarithmic derivative of the Gamma function $\Gamma$. Then, for $\Re s>1$,
\begin{flalign}\label{Landau-Stark}
&& \log|d_K|&=r_1\bigg(\log\pi-\psi\bigg(\frac{s}{2}\bigg)\bigg)
+2r_2(\log2\pi-\psi(s))-\frac{2}{s}-\frac{2}{s-1} &&\nonumber\\
&& &\qquad\qquad\qquad -2\frac{\zeta_K'}{\zeta_K}(s)
+2\sum_{\rho=\beta}\frac{1}{s-\rho}
+2\sum_{\substack{\rho=\beta+i\gamma\\ \gamma>0}}\bigg(\frac{1}{s-\rho}+\frac{1}{s-\overline{\rho}}\bigg)&&\\
\text{and} && 0&=-\frac{r_1}{4}\psi'\bigg(\frac{s}{2}\bigg)
-r_2\psi'(s)+\frac{1}{s^2}+\frac{1}{(s-1)^2}
-\bigg(\frac{\zeta_K'}{\zeta_K}\bigg)'(s) &&\nonumber\\
&& &\qquad\qquad\qquad -\sum_{\rho=\beta}\frac{1}{(s-\rho)^2}
-\sum_{\substack{\rho=\beta+i\gamma\\ \gamma>0}}
\bigg(\frac{1}{(s-\rho)^2}+\frac{1}{(s-\overline{\rho})^2}\bigg).
&&\nonumber
\end{flalign}
\end{theorem}
Stechkin's trick (introduced in \cite{Stechkin*70}) is to replace a logarithmic derivative at
one point by a shifted difference. See \cite{RosserSchoenfeld*75}, \cite{Louboutin*15} and \cite{Louboutin*17} for more details. In the simplest case one compares 
\begin{equation}\label{shifted}
-\frac{\zeta_K'}{\zeta_K}(\sigma+it)
\quad\text{ with }\quad
-\frac{\zeta_K'}{\zeta_K}(\sigma+it)
+\eta \frac{\zeta_K'}{\zeta_K}(\sigma+h+it),
\qquad h\geq 0,\quad 0\leq \eta<1.
\end{equation}
The effect is that the shift introduces in the corresponding Dirichlet coefficients at $n$ the factor $n^{-\sigma}-\eta n^{-(\sigma+h)}= n^{-\sigma}(1-\eta n^{-h})$. Thus, provided $0\leq\eta<  1$ and $h\geq 0$, the prime ideal contribution remains
non-negative on the real axis $s=\sigma>1$, but we have gained flexibility for possible optimisation with the introduction of new parameters $h$ and $\eta$.

Let $0\leq \eta<1$, and let $s,s^*>1$.  We evaluate the
Landau--Stark formulas \eqref{Landau-Stark} at the two points $s$
and $s^*$.  We then multiply the identity evaluated at $s^*$ by
$\eta$, and subtract it from the identity evaluated at $s$.  This
gives the following Stechkin deformation of the Landau--Stark identities, which recovers Theorem \ref{LSI} by setting $\eta=0$.

\begin{theorem}[Landau--Stark--Stechkin] 
Let $0\leq \eta<1$, and let $s,s^*>1$.  Then
\begin{align}
(1-\eta)\log&|d_K|
=r_1\bigg[(1-\eta)\log\pi-\psi\bigg(\frac{s}{2}\bigg)
+\eta\psi\bigg(\frac{s^*}{2}\bigg)\bigg]
-\frac{2}{s}-\frac{2}{s-1}+\frac{2\eta}{s^*}+\frac{2\eta}{s^*-1}
\nonumber\\
&\qquad+2r_2\big[(1-\eta)\log 2\pi-\psi(s)+\eta\psi(s^*)\big]
-2\frac{\zeta_K'}{\zeta_K}(s)+2\eta\frac{\zeta_K'}{\zeta_K}(s^*)
\nonumber\\
&\qquad\quad+2\sum_{\rho=\beta}
\bigg(\frac{1}{s-\rho}-\frac{\eta}{s^*-\rho}\bigg)
+2\sum_{\substack{\rho=\beta+i\gamma\\ \gamma>0}}
\bigg(\frac{1}{s-\rho}+\frac{1}{s-\overline{\rho}}
-\frac{\eta}{s^*-\rho}-\frac{\eta}{s^*-\overline{\rho}}\bigg),
\nonumber\\
\text{ and }\phantom{xxx} 0
&=-\frac{r_1}{4}\bigg[\psi'\bigg(\frac{s}{2}\bigg)
-\eta\psi'\bigg(\frac{s^*}{2}\bigg)\bigg]
-r_2\big[\psi'(s)-\eta\psi'(s^*)\big]
+\frac{1}{s^2}+\frac{1}{(s-1)^2}
\nonumber\\
&\quad-\frac{\eta}{(s^*)^2}-\frac{\eta}{(s^*-1)^2}
-\bigg(\frac{\zeta_K'}{\zeta_K}\bigg)'(s)
+\eta\bigg(\frac{\zeta_K'}{\zeta_K}\bigg)'(s^*)
-\sum_{\rho=\beta}
\bigg(\frac{1}{(s-\rho)^2}-\frac{\eta}{(s^*-\rho)^2}\bigg)
\nonumber\\
&\qquad\quad-\sum_{\substack{\rho=\beta+i\gamma\\ \gamma>0}}
\bigg(\frac{1}{(s-\rho)^2}+\frac{1}{(s-\overline{\rho})^2}
-\frac{\eta}{(s^*-\rho)^2}-\frac{\eta}{(s^*-\overline{\rho})^2}\bigg).
\label{Landau-Stark-Stechkin-Derivative}
\end{align}
\end{theorem}

\section{Low-lying zeros of $\zeta_K$: Hoffstein's argument via Stechkin's trick}\label{peck}

We now adapt Hoffstein's argument to the shifted differences introduced in \eqref{shifted}. The construction uses three identities. The first is the Landau--Stark identity evaluated at $\sigma_0$ and $\sigma_0+h$ for $\sigma_0>1$ and $h\geq 0$; the second compares the Hadamard product \eqref{Hadamard} via \eqref{Lambda} at $\sigma_0$ and $\sigma_1>\sigma_0$, together with their shifted counterparts; and the third is obtained from the differentiated Landau--Stark identity at $\sigma_2$ and $\sigma_2+h$ for $\sigma_2>1$. Those three identities combine to produce the comparison weight $\Phi_{\bsigma}$ defined in \S\ref{weight}.

We put
\begin{equation}\label{delta_K}
 \delta_K=\frac{|d_K|}{2^{2r_2}\pi^{n_K}}.
\end{equation}

\subsection{First identity} Let
\begin{equation*}
 \widetilde{F}_1(s)
 :=
 -\frac{r_1}{2}\psi\bigg(\frac{s}{2}\bigg)
 -r_2\psi(s)-\frac{1}{s}-\frac{1}{s-1}-\frac{\zeta_K'}{\zeta_K}(s).
\end{equation*}

\begin{lemma}\label{Identity:1} 
Let $\sigma_0>1$, $h\geq 0$ and $0\leq\eta<1$. Define
\begin{equation}\label{F1:def}
 \mathcal{F}_1(\sigma_0;h,\eta):=\widetilde{F}_1(\sigma_0)-\eta \widetilde{F}_1(\sigma_0+h).
 \end{equation}
Then
\begin{align*}
\frac{1-\eta}{2}\log \delta_K
=\ &
\mathcal{F}_1(\sigma_0;h,\eta)
+\sum_{\rho=\beta}\bigg(\frac{1}{\sigma_0-\rho}-\frac{\eta}{\sigma_0+h-\rho}\bigg)
\\
&\qquad+
\sum_{\substack{\rho=\beta+i\gamma\\ \gamma>0}}
\bigg(\frac{2(\sigma_0-\beta)}{(\sigma_0-\beta)^2+\gamma^2}-\frac{2\eta(\sigma_0+h-\beta)}{(\sigma_0+h-\beta)^2+\gamma^2}
\bigg).
\end{align*}
\end{lemma} 
\begin{proof}
By the Landau--Stark identity \eqref{Landau-Stark}, we obtain
\begin{align}\label{p1:1}
\frac{1-\eta}{2}\log \delta_K
=\ &\widetilde{F}_1(s)-\eta \widetilde{F}_1(s^*)+\sum_{\rho=\beta}\bigg(\frac{1}{s-\rho}
-\frac{\eta}{s^*-\rho}\bigg)
\nonumber\\
&\qquad+\sum_{\substack{\rho=\beta+i\gamma\\ \gamma>0}}\bigg(\frac{1}{s-\rho}
+\frac{1}{s-\overline{\rho}}-\frac{\eta}{s^*-\rho}-\frac{\eta}{s^*-\overline{\rho}}
\bigg).
\end{align}
Furthermore, we have
\begin{align*}
\frac{1}{s-\rho}+\frac{1}{s-\overline{\rho}}-\frac{\eta}{s^*-\rho}-\frac{\eta}{s^*-\overline{\rho}}=
\frac{2(s-\beta)} {(s-\beta)^2+\gamma^2}-\frac{2\eta(s^*-\beta)}{(s^*-\beta)^2+\gamma^2}.
\end{align*}
 On combining \eqref{F1:def} and \eqref{p1:1} and selecting $s=\sigma_0$, $s^*=\sigma_0+h$, we obtain the result.
 \end{proof}
 
\subsection{Second identity}

Recall equations \eqref{Hadamard} and \eqref{delta_K}. We have the following entire function
\begin{equation}\label{Hadamard2}
s(s-1)\Lambda(s)
=\delta_K^{s/2}\Gamma\bigg(\frac{s}{2}\bigg)^{r_1}\Gamma(s)^{r_2}\zeta_K(s)s(s-1).
\end{equation}
Taking the quotient of
\eqref{Hadamard2} evaluated at $s^*$ by the same equation evaluated at
$s$, we obtain
\begin{equation*}
\frac{s^*(s^*-1)\Lambda(s^*)}{s(s-1)\Lambda(s)}
=\delta_K^{(s^*-s)/2}\bigg(\frac{\Gamma(s^*/2)}{\Gamma(s/2)}\bigg)^{r_1}\bigg(\frac{\Gamma(s^*)}{\Gamma(s)}\bigg)^{r_2}\frac{\zeta_K(s^*)}{\zeta_K(s)}\frac{s^*(s^*-1)}{s(s-1)}.
\end{equation*}
Therefore, taking logarithms gives
\begin{equation*}
\log\bigg(\frac{s^*(s^*-1)\Lambda(s^*)}{s(s-1)\Lambda(s)}\bigg)
=\frac{s^*-s}{2}\log \delta_K+\widetilde{F}_2(s,s^*),
\end{equation*}
where
\begin{equation*}
\widetilde{F}_2(s,s^*)
=r_1\log\bigg(\frac{\Gamma(s^*/2)}{\Gamma(s/2)}\bigg)
+r_2\log\bigg(\frac{\Gamma(s^*)}{\Gamma(s)}\bigg)
+\log\bigg(\frac{\zeta_K(s^*)(s^*-1)}{\zeta_K(s)(s-1)}\bigg)
+\log\bigg(\frac{s^*}{s}\bigg).
\end{equation*}
We now introduce the Stechkin deformation of this identity. 
\begin{lemma}\label{Identity:2} 
Let $\sigma_0>1$, $\sigma_1>\sigma_0$, $h\geq 0$ and $0\leq \eta<1$. Define
\begin{equation}\label{F2:def}
\mathcal{F}_{2}(\sigma_0,\sigma_1;h,\eta)
:=
\widetilde{F}_2(\sigma_0,\sigma_1)
-\eta \widetilde{F}_2(\sigma_0+h,\sigma_1+h).
\end{equation}
Then
\begin{align*}
\frac{(1-\eta)(\sigma_1-\sigma_0)}{2}\log \delta_K+\mathcal{F}_{2}(\sigma_0,\sigma_1;h,\eta)
=\sum_{\rho=\beta}\bigg[\log\bigg(\frac{\sigma_1-\rho}{\sigma_0-\rho}\bigg)-\eta\log\bigg(\frac{\sigma_1+h-\rho}{\sigma_0+h-\rho}\bigg)\bigg]\qquad &\\
+\sum_{\substack{\rho=\beta+i\gamma\\ \gamma>0}}\bigg[\log\bigg(\frac{(\sigma_1-\beta)^2+\gamma^2}{(\sigma_0-\beta)^2+\gamma^2}\bigg)
-\eta\log\bigg(\frac{(\sigma_1+h-\beta)^2+\gamma^2}{(\sigma_0+h-\beta)^2+\gamma^2}\bigg)\bigg].&
\end{align*}
\end{lemma}

\begin{proof}
By the Hadamard product \eqref{pound}, regrouped in
conjugate pairs, we also have
\begin{align}\label{p2:1}
\frac{\sigma_1(\sigma_1-1)\Lambda(\sigma_1)}
{\sigma_0(\sigma_0-1)\Lambda(\sigma_0)}
&=\prod_{\rho=\beta}\frac{1-\sigma_1/\rho}{1-\sigma_0/\rho}\prod_{\substack{\rho=\beta+i\gamma\\ \gamma>0}}
\frac{(1-\sigma_1/\rho)(1-\sigma_1/\overline{\rho})}{(1-\sigma_0/\rho)(1-\sigma_0/\overline{\rho})}\nonumber\\
&=\prod_{\rho=\beta}\frac{\sigma_1-\rho}{\sigma_0-\rho}\prod_{\substack{\rho=\beta+i\gamma\\ \gamma>0}}\frac{(\sigma_1-\beta)^2+\gamma^2}{(\sigma_0-\beta)^2+\gamma^2}.
\end{align}
Indeed, for any real $\sigma$,
\begin{equation*}
\bigg(1-\frac{\sigma}{\rho}\bigg)\bigg(1-\frac{\sigma}{\overline{\rho}}\bigg)
=\frac{(\rho-\sigma)(\overline{\rho}-\sigma)}{\rho\overline{\rho}}
=\frac{(\sigma-\beta)^2+\gamma^2}{\beta^2+\gamma^2}.
\end{equation*}
Thus, taking logarithms in \eqref{p2:1}, we obtain
\begin{equation*}\log\bigg(\frac{\sigma_1(\sigma_1-1)\Lambda(\sigma_1)}{\sigma_0(\sigma_0-1)\Lambda(\sigma_0)}\bigg)
=\sum_{\rho=\beta}\log\bigg(\frac{\sigma_1-\rho}{\sigma_0-\rho}\bigg)+\sum_{\substack{\rho=\beta+i\gamma\\ \gamma>0}}
\log\bigg(\frac{(\sigma_1-\beta)^2+\gamma^2}{(\sigma_0-\beta)^2+\gamma^2}\bigg).
\end{equation*}
Comparing the two expressions gives the following identity
\begin{equation}\label{Identity:2:2}
\frac{\sigma_1-\sigma_0}{2}\log \delta_K+\widetilde{F}_2(\sigma_0,\sigma_1)
=\sum_{\rho=\beta}\log\bigg(\frac{\sigma_1-\rho}{\sigma_0-\rho}\bigg)+\sum_{\substack{\rho=\beta+i\gamma\\ \gamma>0}}
\log\bigg(\frac{(\sigma_1-\beta)^2+\gamma^2}{(\sigma_0-\beta)^2+\gamma^2}\bigg).
\end{equation}

Subtracting $\eta$ times \eqref{Identity:2:2} with
$(\sigma_0,\sigma_1)$ replaced by
$(\sigma_0+h,\sigma_1+h)$, we obtain the result.
\end{proof}

\subsection{Third identity}

Let
\begin{equation*}
\widetilde{F}_3(s):=
-\frac{r_1}{4}\psi'\bigg(\frac{s}{2}\bigg)
-r_2\psi'(s)+\frac{1}{s^2}+\frac{1}{(s-1)^2}
-\bigg(\frac{\zeta_K'}{\zeta_K}\bigg)'(s).
\end{equation*}
The differentiated Stechkin identity \eqref{Landau-Stark-Stechkin-Derivative}
gives
\begin{align}\label{eq:F3def}
\widetilde{F}_3(s)-\eta \widetilde{F}_3(s^*)
=\ &\sum_{\rho=\beta}\bigg(\frac{1}{(s-\rho)^2}-\frac{\eta}{(s^*-\rho)^2}\bigg)
\nonumber\\
&\qquad +\sum_{\substack{\rho=\beta+i\gamma\\ \gamma>0}}\bigg(\frac{1}{(s-\rho)^2}+\frac{1}{(s-\overline{\rho})^2}-\frac{\eta}{(s^*-\rho)^2}-\frac{\eta}{(s^*-\overline{\rho})^2}
\bigg).
\end{align}

We have the following linear shift deformation identity.

\begin{lemma}\label{Identity:3} 
Let $\sigma_2>1$, $h\geq0$ and $0\leq \eta<1$. Define
\begin{equation*}
\mathcal{F}_{3}(\sigma_2;h,\eta)
:=\widetilde{F}_3(\sigma_2)-\eta \widetilde{F}_3(\sigma_2+h).
\end{equation*}
Then
\begin{align*}
\mathcal{F}_{3}(\sigma_2;h,\eta)
=\ &\sum_{\rho=\beta}\bigg(\frac{1}{(\sigma_2-\rho)^2}-\frac{\eta}{(\sigma_2+h-\rho)^2}\bigg)\\
&\qquad +\sum_{\substack{\rho=\beta+i\gamma\\ \gamma>0}}\bigg[\frac{2((\sigma_2-\beta)^2-\gamma^2)}{(\gamma^2+(\sigma_2-\beta)^2)^2}
-\frac{2\eta((\sigma_2+h-\beta)^2-\gamma^2)}{(\gamma^2+(\sigma_2+h-\beta)^2)^2}\bigg].
\end{align*}
\end{lemma}
\begin{proof}
Observe that
\begin{equation*}
\frac{1}{(s-\rho)^2}+\frac{1}{(s-\overline{\rho})^2}-\frac{\eta}{(s^*-\rho)^2}-\frac{\eta}{(s^*-\overline{\rho})^2}
=
\frac{2((s-\beta)^2-\gamma^2)}{(\gamma^2+(s-\beta)^2)^2}
-\frac{2\eta((s^*-\beta)^2-\gamma^2)}{(\gamma^2+(s^*-\beta)^2)^2}.
\end{equation*}
Taking $s=\sigma_2$ and $s^*=\sigma_2+h$ in \eqref{eq:F3def} and summing over $\rho=\beta+i\gamma$ with $\gamma>0$ gives the claimed identity.
\end{proof} 

\subsection{Construction of a comparison weight}\label{weight}

Define
\begin{align*}\label{Phi}
\Phi_{\bsigma}(\beta,\gamma;h,\eta):=\log
\bigg(
\frac{(\sigma_1-\beta)^2+\gamma^2}
{(\sigma_0-\beta)^2+\gamma^2}
\bigg)- \frac{2(\sigma_1-\sigma_0)(\sigma_0-\beta)}
 {(\sigma_0-\beta)^2+\gamma^2}+\frac{(\sigma_1-\sigma_0)^2((\sigma_2-\beta)^2-\gamma^2)}
 {(\gamma^2+(\sigma_2-\beta)^2)^2}&\\
-\eta\log
\bigg(\frac{(\sigma_1+h-\beta)^2+\gamma^2}
{(\sigma_0+h-\beta)^2+\gamma^2}\bigg)
+\frac{2\eta(\sigma_1-\sigma_0)(\sigma_0+h-\beta)}
 {(\sigma_0+h-\beta)^2+\gamma^2}
 -\frac{\eta(\sigma_1-\sigma_0)^2((\sigma_2+h-\beta)^2-\gamma^2)}
 {(\gamma^2+(\sigma_2+h-\beta)^2)^2}\nonumber&.
\end{align*}

Using Lemma \ref{Taylor}, we show that the leading terms in Lemmas \ref{expand:1} and \ref{expand:2} cancel. This is the reason for defining the weight $\Phi_{\bsigma}(\beta,\gamma;h,\eta)$. Indeed, it is useful to analyse the sign of the secondary coefficients, which, unlike the main term, have a chance to be negative (or, rather, having a fixed sign) within a region. 
 
 We will need the following auxiliary lemma.

\begin{lemma}\label{Taylor} 
We have the following bounds.
\begingroup
\setcounter{enumi}{0}
\renewcommand{\theenumi}{\textnormal{\textbf{(\alph{enumi})}}}

\begin{center}
\renewcommand{\arraystretch}{2}
\begin{tabular}{@{}r@{\quad}r@{$\;\leq\;$}l@{}}

\refstepcounter{enumi}\label{taylor:a}\theenumi
& $\displaystyle\quad\quad\qquad\bigg|\log(1+u)-\bigg(u-\frac{u^2}{2}+\frac{u^3}{3}\bigg)\bigg|$
& $\displaystyle\frac{u^4}{4},\qquad 0\leq u\leq 1.$ \\[10pt]

\refstepcounter{enumi}\label{taylor:b}
\theenumi& $\displaystyle\big|(1+u)^{-1}-(1-u+u^2)\big|$
& $\displaystyle u^3,\qquad u\geq 0.$ \\[10pt]

\refstepcounter{enumi}\label{taylor:c}\theenumi
& $\displaystyle\big|(1+u)^{-2}-(1-2u)\big|$
& $\displaystyle 3u^2,\qquad u\geq 0.$ \\[10pt]

\refstepcounter{enumi}\label{taylor:d}\theenumi
& $\displaystyle\big|(1+u)^{-3}-1\big|$
& $\displaystyle 3u,\qquad u\geq 0.$ \\[10pt]

\refstepcounter{enumi}\label{taylor:e}\theenumi
& $\displaystyle\bigg|\frac{1-u}{(1+u)^2}-(1-3u+5u^2)\bigg|$
& $\displaystyle 7u^3,\qquad u\geq 0.$

\end{tabular}
\end{center}
\endgroup
\end{lemma}

\begin{proof}
\ref{taylor:a}. For $|u|\leq 1$, $u\neq -1$, we have the expansion
$\log(1+u)=-\sum_{n\geq 1}(-u)^n/n$.
For $0\leq  u\leq 1$, the terms $u^n/n$ decrease monotonically to $0$. Since we have an alternating series, it follows that
\begin{equation*}
 \bigg| \log(1+u) -\bigg(u-\frac{u^2}{2}+\frac{u^3}{3}\bigg)\bigg|\leq \frac{u^4}{4}.
\end{equation*}
Henceforth, suppose that $u\geq 0$. For \ref{taylor:b}, we observe that
\begin{equation*}
\bigg|\frac{1}{1+u}-(1-u+u^2)\bigg|=\bigg|-\frac{u^3}{1+u}\bigg|\leq  u^3.
\end{equation*}
\ref{taylor:c}$\&$\ref{taylor:d}. We have
\begin{equation*}
 \bigg|\frac{1}{(1+u)^2}-(1-2u)\bigg|=\bigg|u^2\frac{3+2u}{(1+u)^2}\bigg|\leq  3u^2\ \text{ and }\ 
 \bigg|\frac{1}{(1+u)^3}-1\bigg|=\bigg|u\frac{3+3u+u^2}{(1+u)^3}\bigg|\leq  3u.
\end{equation*}
Finally, for \ref{taylor:e}, a direct calculation gives
\begin{equation*}
\bigg|\frac{1-u}{(1+u)^2}-(1-3u+5u^2)\bigg|= \bigg|u^3\frac{7+5u}{(1+u)^2}\bigg|\leq  7u^3.\qedhere
\end{equation*}
\end{proof}

Let $\sigma_1>\sigma_0>1$ and $\sigma_2>1$. 
Observe first that the complex zeta zero contribution in Lemma \ref{Identity:3} satisfies 
\begin{equation}\label{IDD:1}
\sum_{\substack{\rho=\beta+i\gamma\\ \gamma>0}}\bigg[\frac{2((\sigma_2-\beta)^2-\gamma^2)}
{(\gamma^2+(\sigma_2-\beta)^2)^2}
-\frac{2\eta((\sigma_2+h-\beta)^2-\gamma^2)}
{(\gamma^2+(\sigma_2+h-\beta)^2)^2}\bigg] =
\sum_{\substack{\rho=\beta+i\gamma\\ \gamma>0}}\bigg(-\frac{2(1-\eta)}{\gamma^2}
 +O_{\sigma_2,h,\eta}\bigg(\frac{1}{\gamma^4}\bigg)\bigg).
\end{equation}
On the other hand, multiplying the identity of Lemma \ref{Identity:1} by $\sigma_1-\sigma_0$ and substituting into the
identity of Lemma \ref{Identity:2}, we obtain
\begin{align}\label{IDD:2}
(\sigma_1-\sigma_0)\mathcal{F}_1(\sigma_0;h,\eta)
&+\mathcal{F}_{2}(\sigma_0,\sigma_1;h,\eta)\nonumber\\
=\ &\sum_{\rho=\beta}\bigg[\log\bigg(\frac{\sigma_1-\rho}{\sigma_0-\rho}\bigg)-
\eta\log\bigg(\frac{\sigma_1+h-\rho}{\sigma_0+h-\rho}\bigg)-\frac{\sigma_1-\sigma_0}{\sigma_0-\rho}
+\frac{\eta(\sigma_1-\sigma_0)}{\sigma_0+h-\rho}\bigg]\nonumber\\
&\qquad+\sum_{\substack{\rho=\beta+i\gamma\\ \gamma>0}}\bigg[
\log\bigg(\frac{(\sigma_1-\beta)^2+\gamma^2}{(\sigma_0-\beta)^2+\gamma^2}\bigg)
-\eta\log\bigg(\frac{(\sigma_1+h-\beta)^2+\gamma^2}{(\sigma_0+h-\beta)^2+\gamma^2}\bigg)\nonumber\\
&\qquad\qquad-\frac{2(\sigma_0-\beta)(\sigma_1-\sigma_0)}{(\sigma_0-\beta)^2+\gamma^2}+\frac{2\eta(\sigma_0+h-\beta)(\sigma_1-\sigma_0)}{(\sigma_0+h-\beta)^2+\gamma^2}
\bigg].
\end{align}
The reason for doing the above is that the second term in \eqref{IDD:2}, the complex zeta zero summation, removes
the linear term in $(\sigma_1-\sigma_0)$ in the Taylor expansion of the logarithm, leaving a 
leading term that is
quadratic in $\sigma_1-\sigma_0$.
Indeed, we have the following result.
\begin{lemma}\label{expand:1}
Let $\sigma_1>\sigma_0>1$, $0<\beta<1$, $0\leq \eta<1$, $h\geq 0$ and let $\gamma>0$ be sufficiently large so that 
\begin{equation*}
0<\frac{2(\sigma_0-\beta)(\sigma_1-\sigma_0)+(\sigma_1-\sigma_0)^2}{(\sigma_0-\beta)^2+\gamma^2}\leq  1,\quad
0<\frac{2(\sigma_0+h-\beta)(\sigma_1-\sigma_0)+(\sigma_1-\sigma_0)^2}
{(\sigma_0+h-\beta)^2+\gamma^2}
\leq  1.
\end{equation*}
Define 
\begin{align*}
\widetilde{C}_4(\sigma_0,\sigma_1;\beta)&=3(\sigma_0-\beta)^2 +2(\sigma_0-\beta)(\sigma_1-\sigma_0)+\frac{1}{2}(\sigma_1-\sigma_0)^2,\\
\widetilde{C}_6(\sigma_0,\sigma_1;\beta)&=5(\sigma_0-\beta)^4+ \tfrac{20}{3}(\sigma_0-\beta)^3(\sigma_1-\sigma_0) +5(\sigma_0-\beta)^2(\sigma_1-\sigma_0)^2\\
&\qquad+2(\sigma_0-\beta)(\sigma_1-\sigma_0)^3+\tfrac{1}{3}(\sigma_1-\sigma_0)^4,\\
\widetilde{C}_8(\sigma_0,\sigma_1;\beta)&= (\sigma_0+\sigma_1-2\beta)\big[(\sigma_0-\beta)^6+\tfrac{3}{2}(\sigma_1-\sigma_0)(\sigma_0+\sigma_1-2\beta)(\sigma_0-\beta)^4\\
+(\sigma_1&-\sigma_0)^2(\sigma_0+\sigma_1-2\beta)^2(\sigma_0-\beta)^2
+\tfrac{1}{4}(\sigma_1-\sigma_0)^3(\sigma_0+\sigma_1-2\beta)^3\big]+2(\sigma_0-\beta)^7.
\end{align*}
and 
\begin{align*}
C_j(\sigma_0,\sigma_1;\beta;h,\eta)&:=\widetilde{C}_j(\sigma_0,\sigma_1;\beta)-\eta\widetilde{C}_j(\sigma_0+h,\sigma_1+h;\beta),\qquad j\in\{4,6\},\\
C_8(\sigma_0,\sigma_1;\beta;h,\eta)&:=\widetilde{C}_8(\sigma_0,\sigma_1;\beta)+\eta\widetilde{C}_8(\sigma_0+h,\sigma_1+h;\beta).
\end{align*}
Then
\begin{align*}
&\log\bigg(\frac{(\sigma_1-\beta)^2+\gamma^2}{(\sigma_0-\beta)^2+\gamma^2}\bigg)-\eta\log\bigg(\frac{(\sigma_1+h-\beta)^2+\gamma^2}{(\sigma_0+h-\beta)^2+\gamma^2}\bigg)\nonumber\\
&\qquad\qquad\qquad-
 \frac{2(\sigma_0-\beta)(\sigma_1-\sigma_0)}{(\sigma_0-\beta)^2+\gamma^2}+\frac{2\eta(\sigma_0+h-\beta)(\sigma_1-\sigma_0)}{(\sigma_0+h-\beta)^2+\gamma^2}\\
=\ &\frac{(1-\eta)(\sigma_1-\sigma_0)^2}{\gamma^2}
-\frac{(\sigma_1-\sigma_0)^2C_4}{\gamma^4}+
\frac{(\sigma_1-\sigma_0)^2C_6}{\gamma^6}
+O^*\bigg(\frac{(\sigma_1-\sigma_0)C_8}{\gamma^8}\bigg).
\end{align*}
\end{lemma}

\begin{proof}
We have
\begin{align*}
 \log\bigg(\frac{(\sigma_1-\beta)^2+\gamma^2}{(\sigma_0-\beta)^2+\gamma^2}\bigg)
 &=
 \log\bigg(1+ \frac{2(\sigma_0-\beta)(\sigma_1-\sigma_0)+(\sigma_1-\sigma_0)^2}{(\sigma_0-\beta)^2+\gamma^2} \bigg),\\
 -\eta\log\bigg(\frac{(\sigma_1+h-\beta)^2+\gamma^2}{(\sigma_0+h-\beta)^2+\gamma^2}\bigg)
 &=-\eta\log \bigg(1+ \frac{2(\sigma_0+h-\beta)(\sigma_1-\sigma_0)+(\sigma_1-\sigma_0)^2}{(\sigma_0+h-\beta)^2+\gamma^2} \bigg).
\end{align*}
By Lemma \ref{Taylor}\ref{taylor:a},
\begin{align*}
&\log\bigg(1+ \frac{2(\sigma_0-\beta)(\sigma_1-\sigma_0)+(\sigma_1-\sigma_0)^2}{(\sigma_0-\beta)^2+\gamma^2} \bigg)
\\
&=
\frac{2(\sigma_0-\beta)(\sigma_1-\sigma_0)+(\sigma_1-\sigma_0)^2}  {(\sigma_0-\beta)^2+\gamma^2}-\frac{1}{2}
\bigg(\frac{2(\sigma_0-\beta)(\sigma_1-\sigma_0) +(\sigma_1-\sigma_0)^2}{(\sigma_0-\beta)^2+\gamma^2}\bigg)^2
\\
&\quad
+\frac{1}{3}\bigg(\frac{2(\sigma_0-\beta)(\sigma_1-\sigma_0)+ (\sigma_1-\sigma_0)^2} {(\sigma_0-\beta)^2+\gamma^2}
\bigg)^3+O^*\bigg(\frac{1}{4}\bigg(\frac{2(\sigma_0-\beta)(\sigma_1-\sigma_0)+(\sigma_1-\sigma_0)^2 }{(\sigma_0-\beta)^2+\gamma^2}
\bigg)^4\bigg).
\end{align*}
We expand every denominator in powers of $\gamma^{-2}$ using
\begin{equation*}
\frac{1}{(\sigma_0-\beta)^2+\gamma^2}=\frac{1}{\gamma^2}
   \bigg(1+\frac{(\sigma_0-\beta)^2}{\gamma^2}\bigg)^{-1}.
\end{equation*}
More precisely, thanks to Lemma \ref{Taylor}\ref{taylor:b},\ref{taylor:c},\ref{taylor:d}, we have
\begin{align}\label{taylors}\frac{1}{(\sigma_0-\beta)^2+\gamma^2}
 &=\frac{1}{\gamma^2}- \frac{(\sigma_0-\beta)^2}{\gamma^4}+\frac{(\sigma_0-\beta)^4}{\gamma^6}+
O^*\bigg(\frac{(\sigma_0-\beta)^6}{\gamma^8}\bigg)\nonumber\\
\frac{1}{((\sigma_0-\beta)^2+\gamma^2)^2}&= \frac{1}{\gamma^4}-\frac{2(\sigma_0-\beta)^2}{\gamma^6}
+ O^*\bigg(\frac{3(\sigma_0-\beta)^4}{\gamma^8}\bigg),\\
 \frac{1}{((\sigma_0-\beta)^2+\gamma^2)^3}&=
\frac{1}{\gamma^6}+O^*\bigg(\frac{3(\sigma_0-\beta)^2}{\gamma^8}\bigg).\nonumber
\end{align}
Therefore,
\begin{align}\label{series:1}
\frac{ 2(\sigma_0-\beta)(\sigma_1-\sigma_0) +(\sigma_1-\sigma_0)^2}{(\sigma_0-\beta)^2+\gamma^2}=
\frac{2(\sigma_0-\beta)(\sigma_1-\sigma_0)+ (\sigma_1-\sigma_0)^2}{\gamma^2}\qquad\qquad\qquad&
\nonumber\\
-\frac{ (2(\sigma_0-\beta)(\sigma_1-\sigma_0)+ (\sigma_1-\sigma_0)^2)(\sigma_0-\beta)^2}{\gamma^4}
+\frac{ (2(\sigma_0-\beta)(\sigma_1-\sigma_0)+(\sigma_1-\sigma_0)^2)(\sigma_0-\beta)^4}{\gamma^6}&
\nonumber\\
\quad\ \ +O^*\bigg(\frac{ (2(\sigma_0-\beta)(\sigma_1-\sigma_0)+(\sigma_1-\sigma_0)^2)(\sigma_0-\beta)^6 }{\gamma^8}\bigg),&
\end{align}
while
\begin{align}\label{series:2}
\frac{1}{2}
\bigg(\frac{2(\sigma_0-\beta)(\sigma_1-\sigma_0) +(\sigma_1-\sigma_0)^2} {(\sigma_0-\beta)^2+\gamma^2 }
\bigg)^2
=\frac{  (2(\sigma_0-\beta)(\sigma_1-\sigma_0)+(\sigma_1-\sigma_0)^2)^2}{2\gamma^4}\qquad\qquad\qquad&
\nonumber\\
-\frac{(2(\sigma_0-\beta)(\sigma_1-\sigma_0)+ (\sigma_1-\sigma_0)^2 )^2(\sigma_0-\beta)^2}{\gamma^6}
\nonumber\qquad &\\
+\ O^*\bigg(\frac{3(2(\sigma_0-\beta)(\sigma_1-\sigma_0)+ (\sigma_1-\sigma_0)^2)^2(\sigma_0-\beta)^4}{2\gamma^8}
\bigg)&,
\end{align}
and
\begin{align}\label{series:3}
\frac{1}{3}\bigg(\frac{2(\sigma_0-\beta)(\sigma_1-\sigma_0)+(\sigma_1-\sigma_0)^2}{  (\sigma_0-\beta)^2+\gamma^2}\bigg)^3
=\frac{(2(\sigma_0-\beta)(\sigma_1-\sigma_0)+ (\sigma_1-\sigma_0)^2)^3}{3\gamma^6}\qquad\qquad\qquad\quad&
\nonumber\\
+\ O^*\bigg(\frac{(2(\sigma_0-\beta)(\sigma_1-\sigma_0)+(\sigma_1-\sigma_0)^2)^3(\sigma_0-\beta)^2}{\gamma^8}
\bigg).&
\end{align}
Finally, the Taylor remainder satisfies
\begin{equation}\label{series:4}
O^*\bigg(\frac{1}{4}\bigg(\frac{ 2(\sigma_0-\beta)(\sigma_1-\sigma_0) +(\sigma_1-\sigma_0)^2}{ (\sigma_0-\beta)^2+\gamma^2}\bigg)^4\bigg)=
O^*\bigg(\frac{(2(\sigma_0-\beta)(\sigma_1-\sigma_0)+(\sigma_1-\sigma_0)^2)^4}{4\gamma^8}
\bigg).
\end{equation}
Observe that the above estimates hold if we replace $(\sigma_0,\sigma_1)$ by $(\sigma_0+h,\sigma_1+h)$. Combining the preceding four expansions \eqref{series:1},  \eqref{series:2},  \eqref{series:3} and  \eqref{series:4} gives

\begin{align}\label{expansion:1}
&\log\bigg(1+
\frac{2(\sigma_0-\beta)(\sigma_1-\sigma_0)+(\sigma_1-\sigma_0)^2}
{(\sigma_0-\beta)^2+\gamma^2}\bigg)\nonumber\\
&\qquad\qquad\qquad-\eta\log\bigg(1+
\frac{2(\sigma_0+h-\beta)(\sigma_1-\sigma_0)+(\sigma_1-\sigma_0)^2}
{(\sigma_0+h-\beta)^2+\gamma^2}\bigg)\nonumber\\
=\ &(\sigma_1-\sigma_0)\bigg[
\frac{(1-\eta)(\sigma_0+\sigma_1-2\beta)-2\eta h}{\gamma^2}
-\frac{A_4}{\gamma^4}
+\frac{A_6}{\gamma^6}
+O^*\bigg(\frac{A_8}{\gamma^8}\bigg)\bigg],
\end{align}
where 
\begin{align*}
A_j(\sigma_0,\sigma_1;\beta;h,\eta)&:=\widetilde{A}_j(\sigma_0,\sigma_1;\beta)-\eta\widetilde{A}_j(\sigma_0+h,\sigma_1+h;\beta),\qquad j\in\{4,6\},\\
A_8(\sigma_0,\sigma_1;\beta;h,\eta)&:=\widetilde{A}_8(\sigma_0,\sigma_1;\beta)+\eta\widetilde{A}_8(\sigma_0+h,\sigma_1+h;\beta),
\end{align*}
with 
\begin{align*}
\widetilde{A}_4(\sigma_0,\sigma_1;\beta)&=(\sigma_0+\sigma_1-2\beta)\big[(\sigma_0-\beta)^2+ \tfrac{1}{2}(\sigma_1-\sigma_0)
   (\sigma_0+\sigma_1-2\beta)\big],\\
\widetilde{A}_6(\sigma_0,\sigma_1;\beta)&=(\sigma_0+\sigma_1-2\beta) \big[(\sigma_0-\beta)^4
+ (\sigma_1-\sigma_0)( \sigma_0+\sigma_1-2\beta)(\sigma_0-\beta)^2\\
    &\qquad\qquad\qquad\qquad\qquad+\tfrac{1}{3}(\sigma_1-\sigma_0)^2 (\sigma_0+\sigma_1-2\beta)^2\big],\\
\widetilde{A}_8(\sigma_0,\sigma_1;\beta)&=(\sigma_0+\sigma_1-2\beta)\big[(\sigma_0-\beta)^6
+ \tfrac{3}{2}(\sigma_1-\sigma_0)(\sigma_0+\sigma_1-2\beta)(\sigma_0-\beta)^4\\
&\quad+
        (\sigma_1-\sigma_0)^2(\sigma_0+\sigma_1-2\beta)^2(\sigma_0-\beta)^2
+ \tfrac{1}{4}(\sigma_1-\sigma_0)^3(\sigma_0+\sigma_1-2\beta)^3\big].
\end{align*}
On the other hand, by considering expressions \eqref{taylors} with $\sigma_0$ and also $\sigma_0+h$, we derive
\begin{align}\label{expansion:2}
 \frac{2(\sigma_0-\beta)(\sigma_1-\sigma_0)}{(\sigma_0-\beta)^2+\gamma^2}&-\frac{2\eta(\sigma_0+h-\beta)(\sigma_1-\sigma_0)}
 {(\sigma_0+h-\beta)^2+\gamma^2}\nonumber\\
&=(\sigma_1-\sigma_0)\bigg[\frac{2(1-\eta)(\sigma_0-\beta)-2\eta h}{\gamma^2}
         -\frac{B_4}{\gamma^4} + \frac{B_6}{\gamma^6}+ O^*\bigg(\frac{B_8}{\gamma^8}\bigg)\bigg],
\end{align}
where $B_4=B_4(\sigma_0;\beta;h,\eta)$, $B_6=B_6(\sigma_0;\beta;h,\eta)$ and $B_8=B_8(\sigma_0;\beta;h,\eta)$ are defined as
\begin{align*}
 B_4&=2(\sigma_0-\beta)^3-2\eta(\sigma_0+h-\beta)^3,\qquad\qquad B_6=2(\sigma_0-\beta)^5-2\eta(\sigma_0+h-\beta)^5,\\
&\qquad\qquad\qquad\qquad B_8=2(\sigma_0-\beta)^7+2\eta(\sigma_0+h-\beta)^7.
\end{align*}
Observe now that $A_4-B_4=(\sigma_1-\sigma_0)C_4$, $A_6-B_6=(\sigma_1-\sigma_0)C_6$,
$A_8+B_8=C_8$.
Subtracting the two expansions \eqref{expansion:1} and \eqref{expansion:2} gives the result.
\end{proof}

Thanks to Lemmas \ref{expand:1} and \ref{expand:2}, identities \eqref{IDD:1} and \eqref{IDD:2} are asymptotically equivalent if we multiply identity \eqref{IDD:1} by $-\frac{1}{2}(\sigma_1-\sigma_0)^2$. More precisely, we have the following result.
\begin{lemma}\label{expand:2} 
Let $\sigma_2>1$, $\sigma_1>\sigma_0>1$, $0<\beta<1$,
$h\ge0$, $0\le\eta<1$, and $\gamma>0$. Then
\begin{align*}
&\frac{(\sigma_1-\sigma_0)^2}{2}\bigg[\frac{2((\sigma_2-\beta)^2-\gamma^2)}
{(\gamma^2+(\sigma_2-\beta)^2)^2}
-\frac{2\eta((\sigma_2+h-\beta)^2-\gamma^2)}
{(\gamma^2+(\sigma_2+h-\beta)^2)^2}\bigg]\\
&\qquad\qquad=-\frac{(1-\eta)(\sigma_1-\sigma_0)^2}{\gamma^2}+\frac{(\sigma_1-\sigma_0)^2D_4}{\gamma^4}-\frac{(\sigma_1-\sigma_0)^2D_6}{\gamma^6}
+O^*\bigg(\frac{(\sigma_1-\sigma_0)D_8}{\gamma^8}\bigg).
\end{align*}
where $D_4=D_4(\sigma_2;\beta;h,\eta)$, $D_6=D_6(\sigma_2;\beta;h,\eta)$ and $D_8=D_8(\bsigma;\beta;h,\eta)$ are defined as
\begin{align*}
D_4&=3(\sigma_2-\beta)^2-3\eta(\sigma_2+h-\beta)^2,\qquad\qquad D_6=5(\sigma_2-\beta)^4-5\eta(\sigma_2+h-\beta)^4,\\
&\qquad\qquad\qquad\qquad D_8=7(\sigma_1-\sigma_0)(\sigma_2-\beta)^6+7\eta(\sigma_1-\sigma_0)(\sigma_2+h-\beta)^6.
\end{align*}
\end{lemma}
\begin{proof}
We have the following identities
\begin{align*}
\frac{2((\sigma_2-\beta)^2-\gamma^2)}
     {(\gamma^2+(\sigma_2-\beta)^2)^2}
&=-\frac{2}{\gamma^2}\frac{1-(\sigma_2-\beta)^2/\gamma^2}{
(1+(\sigma_2-\beta)^2/\gamma^2)^2},\\
\frac{2((\sigma_2+h-\beta)^2-\gamma^2)}{(\gamma^2+(\sigma_2+h-\beta)^2)^2}
&=-\frac{2}{\gamma^2}\frac{1-(\sigma_2+h-\beta)^2/\gamma^2}{(1+(\sigma_2+h-\beta)^2/\gamma^2)^2}.
\end{align*}
Applying Lemma \ref{Taylor}\ref{taylor:e} with $u=(\sigma_2-\beta)^2/\gamma^2$, $u=(\sigma_2+h-\beta)^2/\gamma^2$ and multiplying by $(\sigma_1-\sigma_0)^2/2$, we derive the result.
\end{proof}

\section{A symmetrised weight}\label{ton}
We have the following identity on the zeros of the Dedekind zeta function $\zeta_K$
\begin{align}
\label{Identity:5}
& \sum_{\substack{\rho=\beta+i\gamma\\ \gamma>0}}\Phi_{\bsigma}(\beta,\gamma;h,\eta)
+\frac{1}{2}\sum_{\substack{\rho=\beta\\}}\Phi_{\bsigma}(\beta,0;h,\eta)\nonumber\\
&\qquad\qquad= (\sigma_1-\sigma_0) 
\mathcal{F}_1(\sigma_0;h,\eta)
 +
 \mathcal{F}_2(\sigma_0,\sigma_1;h,\eta)+\frac{1}{2}(\sigma_1-\sigma_0)^2\mathcal{F}_3(\sigma_2;h,\eta).
\end{align}
We now use the symmetry of the non-trivial zeros under $\rho\longmapsto 1-\rho$.
For $1/2\leq  \beta<1$ and $\gamma\geq 0$, define the symmetrised function
\begin{equation}
\label{Omega}
\Omega_{\bsigma}(\beta,\gamma;h,\eta):=
\Phi_{\bsigma}(\beta,\gamma;h,\eta) + \Phi_{\bsigma}(1-\beta,\gamma;h,\eta).
\end{equation}
In other words, $\Omega_{\bsigma}$ is the contribution of the pair of
zeros $\beta+i\gamma$, $1-\beta+i\gamma$,
after the conjugate zeros have already been grouped.
With this convention, the identity \eqref{Identity:5} may be rewritten
as
\begin{align}\label{eq:Phi-sym-identity}
&\sum_{\substack{\rho=\beta+i\gamma\\ \frac{1}{2}<\beta<1,\ \gamma>0}}
\Omega_{\bsigma}(\beta,\gamma;h,\eta)+\frac{1}{2}
\sum_{\substack{\rho=\frac{1}{2}+i\gamma\\ \gamma>0}}
\Omega_{\bsigma}\!\bigg(\frac{1}{2},\gamma;h,\eta\bigg)
\nonumber\\
&\qquad\qquad\qquad\qquad
+\frac{1}{2}\sum_{\substack{\rho=\beta\\
\frac{1}{2}<\beta<1}}\Omega_{\bsigma}(\beta,0;h,\eta)
+\frac{1}{4}\sum_{\substack{\rho=\frac{1}{2}}}
\Omega_{\bsigma}\!\bigg(\frac{1}{2},0;h,\eta\bigg)\nonumber\\
&\qquad=(\sigma_1-\sigma_0)
\mathcal{F}_1(\sigma_0;h,\eta)+
\mathcal{F}_2(\sigma_0,\sigma_1;h,\eta)+
\frac{1}{2}(\sigma_1-\sigma_0)^2
\mathcal{F}_3(\sigma_2;h,\eta).
\end{align}
Combining Lemmas \ref{expand:1} and \ref{expand:2}, we obtain
\begin{align*}
\Phi_{\bsigma}(\beta,\gamma;h,\eta)
&=-\frac{(\sigma_1-\sigma_0)^2E_4}{\gamma^4}+
\frac{(\sigma_1-\sigma_0)^2E_6}{\gamma^6}+O^*\bigg(\frac{(\sigma_1-\sigma_0)E_8}{\gamma^8}\bigg),
\end{align*}
where
\begin{align}\label{E_i:def}
E_j=E_j(\bsigma;\beta;h,\eta)&:=C_j(\sigma_0,\sigma_1;\beta;h,\eta)-D_j(\sigma_2;\beta;h,\eta),\qquad j\in\{4,6\},\\
E_8=E_8(\bsigma;\beta;h,\eta)&:=C_8(\sigma_0,\sigma_1;\beta;h,\eta)+D_8(\bsigma;\beta;h,\eta).\nonumber
\end{align}

\subsection{Low-zero argument}\label{Argument:1}
First, one wants to prove that the high zeta zero contribution is non-positive.
That is, for a given threshold $T>0$ and fixed parameters $\sigma_1>\sigma_0>1$, $\sigma_2>1$, $h\geq 0$ and $0\leq\eta<1$, one wants
\begin{equation}
\label{Hoffstein:negative}
\sum_{\substack{\rho=\beta+i\gamma\\
\frac{1}{2}<\beta<1,\ \gamma\geq T}}
\Omega_{\bsigma}(\beta,\gamma;h,\eta)
+
\frac{1}{2}\sum_{\substack{\rho=\frac{1}{2}+i\gamma\\ \gamma\geq T}}
\Omega_{\bsigma}\!\bigg(\frac{1}{2},\gamma;h,\eta\bigg)
\leq  0.
\end{equation}
By definition \eqref{Omega}, the above is implied by the pointwise inequality
\begin{equation*}
 \Omega_{\bsigma}(\beta,\gamma;h,\eta)
 \le 0\qquad\text{ for }\quad 1/2\leq \beta<1,\ h\geq 0,\ 0\leq \eta<1,\ \gamma\geq T.
 \end{equation*}
 Second, one wants the total sum in \eqref{Identity:5} to be positive. In other words, one wants
\begin{equation}
\label{Hoffstein:positive}
(\sigma_1-\sigma_0) 
\mathcal{F}_1(\sigma_0;h,\eta)
 +
 \mathcal{F}_2(\sigma_0,\sigma_1;h,\eta)+\frac{1}{2}(\sigma_1-\sigma_0)^2\mathcal{F}_3(\sigma_2;h,\eta)
 >0.
\end{equation}

If \eqref{Hoffstein:negative} and
\eqref{Hoffstein:positive} hold, then the total zero contribution in
\eqref{Identity:5} is positive, whereas the zeros with $\gamma\ge T$
contribute non-positively. Therefore some positive contribution must come
from zeros with $0\leq \gamma<T$. Hence $\zeta_K(s)$ has at least one
zero with $|\gamma|<T$, so that $h(K)<T$.

\subsection{Quantitative argument}\label{quan}

Moreover, suppose that, for some constant $c_0>0$, one has
\begin{equation*}
(\sigma_1-\sigma_0)\mathcal{F}_1(\sigma_0;h,\eta)
+\mathcal{F}_2(\sigma_0,\sigma_1;h,\eta)
+\frac{1}{2}(\sigma_1-\sigma_0)^2
\mathcal{F}_3(\sigma_2;h,\eta)
\ge c_0 n_K-O(1),
\end{equation*}
where the implied constant in $O(1)$ does not depend on $K$.
Assume also that \eqref{Hoffstein:negative} holds. 

Since $\Omega_{\bsigma}(\beta,\gamma;h,\eta)$ is continuous on the compact
box $1/2\leq \beta\leq 1$, $0\leq \gamma\leq  T$,
there exists a constant $M_T=M_T(\bsigma;h,\eta)$, depending only on
$T,\bsigma,h,\eta$, such that
\begin{equation*} 
\max\bigg\{0,\max_{(\beta,\gamma)\in[1/2,1]\times[0,T]}\Omega_{\bsigma}(\beta,\gamma;h,\eta)\bigg\}<M_T.
\end{equation*}
Then the low-zero contribution in \eqref{eq:Phi-sym-identity} satisfies
\begin{align*}
\sum_{\substack{\rho=\beta+i\gamma\\
\frac{1}{2}<\beta<1,\ 0<\gamma<T}}\Omega_{\bsigma}(\beta,\gamma;h,\eta)&
+\frac{1}{2}\sum_{\substack{\rho=\frac{1}{2}+i\gamma\\ 0<\gamma<T}}\Omega_{\bsigma}\!\bigg(\frac{1}{2},\gamma;h,\eta\bigg)\\
&\qquad\qquad+
\frac{1}{2}\sum_{\substack{\rho=\beta\\ \frac{1}{2}<\beta<1}}\Omega_{\bsigma}(\beta,0;h,\eta)
+\frac{1}{4}\sum_{\substack{\rho=\frac{1}{2}}}\Omega_{\bsigma}\!\bigg(\frac{1}{2},0;h,\eta\bigg)\\
\leq\frac{M_T}{2}&\bigg(\sum_{\substack{\rho=\beta+i\gamma\\
\frac{1}{2}<\beta<1,\ 0<\gamma<T}}2
+\sum_{\substack{\rho=\frac{1}{2}+i\gamma\\ 0<\gamma<T}}1
+\sum_{\substack{\rho=\beta\\ \frac{1}{2}<\beta<1}}1
+\frac{1}{2}\sum_{\substack{\rho=\frac{1}{2}}}1\bigg)\\
\leq & \frac{M_T}{2}\bigg(\sum_{\substack{\rho=\beta+i\gamma\\
\frac{1}{2}<\beta<1,\ 0<\gamma\leq T}}2
+\sum_{\substack{\rho=\frac{1}{2}+i\gamma\\ 0<\gamma\leq T}}1
+\sum_{\substack{\rho=\beta\\ \frac{1}{2}<\beta<1}}2
+\sum_{\substack{\rho=\frac{1}{2}}}1\bigg)=\frac{M_T}{2} N_K(T),
\end{align*}
where $N_K(T)$ is defined in \eqref{N_K}.
Thereupon, by \eqref{eq:Phi-sym-identity} and
\eqref{Hoffstein:negative}, we obtain
$c_0 n_K-O(1)\leq  \frac{1}{2}M_T N_K(T)$ so that
\begin{equation*}
N_K(T)\geq\frac{2}{M_T}(c_0 n_K-O(1))\gg_{\bsigma,h,\eta,T} n_K.
\end{equation*}
Thus the low-zero argument in \S\ref{Argument:1} gives not only the existence of one zero with
$0\leq \gamma<T$, but a positive proportion of $n_K$ of such zeros. 

\section{Optimisation problems}\label{rod}

Recall definitions \eqref{F1:def}, \eqref{F2:def} and \eqref{eq:F3def}. Write
\begin{align}\label{expression}
&(\sigma_1-\sigma_0) 
\mathcal{F}_1(\sigma_0;h,\eta)
 +\mathcal{F}_2(\sigma_0,\sigma_1;h,\eta)+\frac{1}{2}(\sigma_1-\sigma_0)^2\mathcal{F}_3(\sigma_2;h,\eta)\nonumber\\
&\qquad\qquad\qquad = r_1\Delta_{\mathbb{R}}(\bsigma;h,\eta)+r_2\Delta_{\mathbb{C}}(\bsigma;h,\eta)+\Delta_O(\bsigma;h,\eta)+P(\bsigma;h,\eta),
 \end{align} 
 where
\begin{align}\label{def:Delta_R}
 \Delta_{\mathbb{R}}(\bsigma;h,\eta)&=-\frac{1}{2}(\sigma_1-\sigma_0)\bigg[\psi\bigg(\frac{\sigma_0}{2}\bigg)-\eta\psi\bigg(\frac{\sigma_0+h}{2}\bigg)\bigg]+\log\bigg(\frac{\Gamma(\sigma_1/2)}{\Gamma(\sigma_0/2)}\bigg)\nonumber\\
&\qquad\qquad-\eta\log\bigg(\frac{\Gamma((\sigma_1+h)/2)}{\Gamma((\sigma_0+h)/2)}\bigg)
-\frac{1}{8}(\sigma_1-\sigma_0)^2\bigg[\psi'\bigg(\frac{\sigma_2}{2}\bigg)-\eta\psi'\bigg(\frac{\sigma_2+h}{2}\bigg)\bigg],\\
\label{def:Delta_C}
 \Delta_{\mathbb{C}}(\bsigma;h,\eta)&=-(\sigma_1-\sigma_0)\big[\psi(\sigma_0)-\eta\psi(\sigma_0+h)\big]
 +\log\bigg(\frac{\Gamma(\sigma_1)}
{\Gamma(\sigma_0)}\bigg)\nonumber\\
&\qquad\qquad-\eta\log\bigg(\frac{\Gamma(\sigma_1+h)}
{\Gamma(\sigma_0+h)}\bigg)
-\frac{1}{2}(\sigma_1-\sigma_0)^2\big[\psi'(\sigma_2)-\eta\psi'(\sigma_2+h)\big],\\
\label{def:Delta_O}
 \Delta_O(\bsigma;h,\eta)&=-(\sigma_1-\sigma_0)\bigg(\frac{1}{\sigma_0} +\frac{1}{\sigma_0-1}-\frac{\eta}{\sigma_0+h}
-\frac{\eta}{\sigma_0+h-1}\bigg)\nonumber\\
&\qquad+\frac{1}{2}(\sigma_1-\sigma_0)^2\bigg(
 \frac{1}{\sigma_2^2}+\frac{1}{(\sigma_2-1)^2}-\frac{\eta}{(\sigma_2+h)^2}-\frac{\eta}{(\sigma_2+h-1)^2}\bigg)\nonumber\\
 &\qquad\qquad+\log\bigg(\frac{\sigma_1(\sigma_1-1)}{\sigma_0(\sigma_0-1)}\bigg)-\eta\log\bigg(\frac{(\sigma_1+h)(\sigma_1+h-1)}{(\sigma_0+h)(\sigma_0+h-1)}\bigg),
  \end{align}
\begin{align}
\label{def:P}
P(\bsigma;h,\eta)
&=-(\sigma_1-\sigma_0)\bigg[\frac{\zeta_K'}{\zeta_K}(\sigma_0)-\eta\frac{\zeta_K'}{\zeta_K}(\sigma_0+h)\bigg]+\log\bigg(\frac{\zeta_K(\sigma_1)}{\zeta_K(\sigma_0)}\bigg)\nonumber\\
&\qquad-\eta\log\bigg(\frac{\zeta_K(\sigma_1+h)}{\zeta_K(\sigma_0+h)}\bigg)-\frac{1}{2}(\sigma_1-\sigma_0)^2\bigg[\bigg(\frac{\zeta_K'}{\zeta_K}\bigg)'(\sigma_2)-\eta\bigg(\frac{\zeta_K'}{\zeta_K}\bigg)'(\sigma_2+h)\bigg].
\end{align}
Note that $P(\bsigma;h,\eta)$ corresponds to the terms that involve $\zeta_K$. Moreover, since $\sigma_1>\sigma_0>1$, $\sigma_2>1$, by Lemma \ref{ideals}, the above expression corresponds to the  prime ideal contribution
\begin{flalign}
&& P&(\bsigma;h,\eta)
=\sum_{\mathfrak p\in P_K^*}\Xi_{\bsigma,h,\eta}(N\mathfrak p),
&&\nonumber\\
\text{where} &&
\Xi_{\bsigma,h,\eta}\colon u\geq 2\mapsto{}&(\sigma_1-\sigma_0)\bigg[
\frac{\log u}{u^{\sigma_0}-1}-\frac{\eta\log u}{u^{\sigma_0+h}-1}\bigg]
+\log\bigg(\frac{1-u^{-\sigma_0}}{1-u^{-\sigma_1}}\bigg)
&&\nonumber\\
&& -\eta\log\bigg(&\frac{1-u^{-(\sigma_0+h)}}{1-u^{-(\sigma_1+h)}}\bigg)
-\frac{(\sigma_1-\sigma_0)^2}{2}\bigg[
\frac{u^{\sigma_2}\log^2 u}{(u^{\sigma_2}-1)^2}
-\frac{\eta u^{\sigma_2+h}\log^2 u}{(u^{\sigma_2+h}-1)^2}\bigg].
&&\label{E:def}
\end{flalign}
Suppose that for given parameters $(\bsigma,h,\eta)$ the function \eqref{E:def} is non-negative for every $u\ge2$. Then the whole function $P(\bsigma;h,\eta)$ may be discarded by positivity when aiming for a lower bound for
\eqref{Hoffstein:positive}. After this discard, the remaining contribution to \eqref{expression} is linear
in $\{r_1,r_2\}$, while the contribution $\Delta_O(\bsigma;h,\eta)$  is bounded independently of $n_K$. If we succeed in proving that $\Delta_{\mathbb{R}}(\bsigma;h,\eta)>0$ and $\Delta_{\mathbb{C}}(\bsigma;h,\eta)>0$, the right-hand side
of expression \eqref{expression} is positive once $n_K=r_1+2r_2$ is sufficiently large.

With the help of the above observations, we have the following optimisation problem.

\begin{problem}\label{Optimisation:A}
 
Find parameters $(\bsigma,h,\eta,T)$ such that 
\begin{enumerate}[
    label=\textnormal{\textbf{(\roman*)}},
    ref=\textnormal{\textbf{(\roman*)}},
    leftmargin=3em,
    itemsep=6pt,
    topsep=6pt,
    parsep=0pt
]

\item\label{OptA:range}\qquad
$\min\{\sigma_0,\sigma_2\}>1,\quad \sigma_1>\sigma_0,\quad 0\leq \eta<1,\quad h\geq 0,\quad T>0,$

\item\label{OptA:Omega}\qquad
$\Omega_{\bsigma}(\beta,\gamma;h,\eta)\leq  0
\quad\text{ for }\quad 1/2\leq \beta<1,\ \gamma\geq T,$

\item\label{OptA:Xi}\qquad
$\Xi_{\bsigma,h,\eta}(u)\geq 0
\quad\text{ for }\quad u\geq 2,$

\item\label{OptA:Delta}\qquad
$\Delta_{\mathbb R}(\bsigma;h,\eta)>0,\quad
\Delta_{\mathbb C}(\bsigma;h,\eta)>0.$\\

\end{enumerate}
\end{problem}

Note that the symmetrised weight pointwise inequality is weaker than the unsymmetrised pointwise version, thus allowing a broader range of quadruples $(\bsigma,T)$.

\subsection{Negativity of the weight function - Condition \ref{Optimisation:A}\ref{OptA:Omega}} 
Recall definitions \eqref{Omega} and \eqref{E_i:def}. We have
\begin{align}\label{exppp}
\Omega_{\bsigma}(\beta,\gamma;h,\eta)
&= \Phi_{\bsigma}(\beta,\gamma;h,\eta) +\Phi_{\bsigma}(1-\beta,\gamma;h,\eta)\nonumber\\
&=-\frac{(\sigma_1-\sigma_0)^2J_4}{\gamma^4}+\frac{(\sigma_1-\sigma_0)^2J_6}{\gamma^6}+O^*\bigg(\frac{(\sigma_1-\sigma_0)J_8}{\gamma^8}\bigg)
\end{align}
where
\begin{equation*}
J_j=J_j(\bsigma;\beta;h,\eta):=E_j(\bsigma;\beta;h,\eta)+E_j(\bsigma;1-\beta;h,\eta),\qquad j\in\{4,6,8\}.
\end{equation*}
Condition \ref{Optimisation:A}\ref{OptA:Omega} will be eventually achieved if $J_4>0$ and $\gamma\geq T\to\infty$.

Note that $J_4$ simplifies exactly into a quantity independent from $\beta$. Namely,
\begin{align*}
J_4(\bsigma;h,\eta)
=\ &(1-\eta)(3\sigma_0^2+2\sigma_0\sigma_1-4\sigma_0+\sigma_1^2-2\sigma_1-6\sigma_2^2+6\sigma_2)-4\eta h
(2\sigma_0+\sigma_1-3\sigma_2).
\end{align*}
Furthermore, since $J_8\geq 0$, in order to prove Condition \ref{Optimisation:A}\ref{OptA:Omega}, it is sufficient that
\begin{equation}\label{condition1}
\max\bigg\{\frac{(\sigma_1-\sigma_0)^2J_6}{\gamma^6}+\frac{(\sigma_1-\sigma_0)J_8}{\gamma^8}\ ,\ 0\bigg\}\leq \frac{(\sigma_1-\sigma_0)^2J_4}{\gamma^4},\qquad J_4>0,
\end{equation}
uniformly in $1/2\leq \beta<1$ and $\gamma\geq T$. 
Now, for $\gamma\geq T_0$, we have $0<1/\gamma^2\leq  1/T_0^2$. Hence
\begin{equation*}
\Omega_{\bsigma}(\beta,\gamma;h,\eta)
\leq \frac{(\sigma_1-\sigma_0)^2}{\gamma^4}\bigg(-J_4+\frac{J_6}{\gamma^2}+\frac{J_8}{(\sigma_1-\sigma_0)\gamma^4}\bigg).
\end{equation*}
Since $J_8\geq0$, the function $u\mapsto J_6u+J_8u^2/(\sigma_1-\sigma_0)$
is convex on $[0,1/T_0^2]$. Therefore
\begin{equation}\label{convex}
J_6u+\frac{J_8}{\sigma_1-\sigma_0}u^2
\leq 
\max\bigg\{0,\,\frac{J_6}{T_0^2}+\frac{J_8}{(\sigma_1-\sigma_0) T_0^4}\bigg\}.
\end{equation}
We are going to verify \eqref{condition1} analytically for $\gamma\geq T_0$, for a threshold $T_0>T$. Notice that for $\gamma\geq T_0$, thanks to \eqref{convex}, condition \eqref{condition1} is satisfied if
\begin{equation*}
\max\bigg\{\frac{J^*_6}{T_0^2}+\frac{J^*_8}{(\sigma_1-\sigma_0)T_0^4}\ ,\ 0\bigg\}\leq  J_4,
\end{equation*}
where, by using the inequalities $1/2\leq \beta<1<\min\{\sigma_0,\sigma_2\}$, $\sigma_0<\sigma_1$, we have
\begin{equation*}
J_6(\bsigma;\beta;h,\eta)\leq  J_6^*(\bsigma;h,\eta),\qquad J_8(\bsigma;\beta;h,\eta)\leq  J_8^*(\bsigma;h,\eta),
\end{equation*}
where $J_6^*(\bsigma;h,\eta)$ is defined as
\begin{align*}
&
5(\sigma_0-\tfrac{1}{2})^4+\tfrac{20}{3}(\sigma_0-\tfrac{1}{2})^3(\sigma_1-\sigma_0)
+5(\sigma_0-\tfrac{1}{2})^2(\sigma_1-\sigma_0)^2
+2(\sigma_0-\tfrac{1}{2})(\sigma_1-\sigma_0)^3
\\
&+\tfrac{2}{3}(\sigma_1-\sigma_0)^4
-5(\sigma_2-1)^4+5\sigma_0^4
+\tfrac{20}{3}\sigma_0^3(\sigma_1-\sigma_0)
+5\sigma_0^2(\sigma_1-\sigma_0)^2
+2\sigma_0(\sigma_1-\sigma_0)^3
\\
&\quad -5(\sigma_2-\tfrac{1}{2})^4-\tfrac{2}{3}\eta(\sigma_1-\sigma_0)^4
+5\eta(\sigma_2+h-\tfrac{1}{2})^4
+5\eta(\sigma_2+h)^4-\eta\big[
5(\sigma_0+h-1)^4\\
&
\qquad+\tfrac{20}{3}(\sigma_0+h-1)^3(\sigma_1-\sigma_0)
+5(\sigma_0+h-1)^2(\sigma_1-\sigma_0)^2+2(\sigma_0+h-1)(\sigma_1-\sigma_0)^3\big]
\\
&\qquad\quad -\eta\big[5(\sigma_0+h-\tfrac{1}{2})^4
+\tfrac{20}{3}(\sigma_0+h-\tfrac{1}{2})^3(\sigma_1-\sigma_0)
+5(\sigma_0+h-\tfrac{1}{2})^2(\sigma_1-\sigma_0)^2
\\
&\qquad\qquad\quad
+2(\sigma_0+h-\tfrac{1}{2})(\sigma_1-\sigma_0)^3\big],
\end{align*}
and $J_8^*(\bsigma;h,\eta)$ is defined as
\begin{align*}
&
(\sigma_0+\sigma_1-1)
\big[(\sigma_0-\tfrac{1}{2})^6+\tfrac{3}{2}(\sigma_1-\sigma_0)(\sigma_0+\sigma_1-1)(\sigma_0-\tfrac{1}{2})^4
\\
&\qquad\qquad\qquad\qquad
+(\sigma_1-\sigma_0)^2(\sigma_0+\sigma_1-1)^2(\sigma_0-\tfrac{1}{2})^2
+\tfrac{1}{4}(\sigma_1-\sigma_0)^3(\sigma_0+\sigma_1-1)^3\big]
\\
&
+(\sigma_0+\sigma_1)\big[\sigma_0^6+\tfrac{3}{2}(\sigma_1-\sigma_0)(\sigma_0+\sigma_1)\sigma_0^4
+(\sigma_1-\sigma_0)^2(\sigma_0+\sigma_1)^2\sigma_0^2+\tfrac{1}{4}(\sigma_1-\sigma_0)^3(\sigma_0+\sigma_1)^3\big]
\\
&
\quad +2(\sigma_0-\tfrac{1}{2})^7+2\sigma_0^7+7(\sigma_1-\sigma_0)(\sigma_2-\tfrac{1}{2})^6
+7(\sigma_1-\sigma_0)\sigma_2^6
\\
&
+\eta(\sigma_0+\sigma_1+2h-1)
\big[(\sigma_0+h-\tfrac{1}{2})^6+\tfrac{3}{2}(\sigma_1-\sigma_0)(\sigma_0+\sigma_1+2h-1)(\sigma_0+h-\tfrac{1}{2})^4
\\
&\qquad\qquad\quad+(\sigma_1-\sigma_0)^2(\sigma_0+\sigma_1+2h-1)^2(\sigma_0+h-\tfrac{1}{2})^2+\tfrac{1}{4}(\sigma_1-\sigma_0)^3(\sigma_0+\sigma_1+2h-1)^3\big]
\\
&
+\eta(\sigma_0+\sigma_1+2h)\big[(\sigma_0+h)^6+\tfrac{3}{2}(\sigma_1-\sigma_0)(\sigma_0+\sigma_1+2h)(\sigma_0+h)^4
\\
&\qquad\qquad\quad+(\sigma_1-\sigma_0)^2(\sigma_0+\sigma_1+2h)^2(\sigma_0+h)^2+\tfrac{1}{4}(\sigma_1-\sigma_0)^3(\sigma_0+\sigma_1+2h)^3\big]
\\
&
+2\eta(\sigma_0+h-\tfrac{1}{2})^7+2\eta(\sigma_0+h)^7+7\eta(\sigma_1-\sigma_0)(\sigma_2+h-\tfrac{1}{2})^6
+7\eta(\sigma_1-\sigma_0)(\sigma_2+h)^6.
\end{align*}

On the other hand, in order to derive \eqref{exppp}, we have implicitly used Lemma \ref{expand:1}. Thus we need to make sure that its conditions hold. Indeed, since $\beta\leq 1<\sigma_0$, for $\gamma\geq T_0$, it suffices to have
\begin{align*}
&\max\bigg\{\frac{2(\sigma_0-\beta)(\sigma_1-\sigma_0)+(\sigma_1-\sigma_0)^2}{(\sigma_0-\beta)^2+\gamma^2},
\frac{2(\sigma_0+h-\beta)(\sigma_1-\sigma_0)+(\sigma_1-\sigma_0)^2}
{(\sigma_0+h-\beta)^2+\gamma^2}\bigg\}\\
&\qquad\qquad\qquad\qquad\qquad\leq\frac{2(\sigma_0+h)(\sigma_1-\sigma_0)+(\sigma_1-\sigma_0)^2}
{T_0^2}
\leq  1.
\end{align*}
All in all, we have the next simplified optimisation problem, which implies Problem \ref{Optimisation:A}. 

\begin{problem}\label{Optimisation:B}
 
Find parameters $(\bsigma,h,\eta,T,T_0)$ such that 
\begin{enumerate}[
    label=\textnormal{\textbf{(\roman*)}},
    ref=\textnormal{\textbf{(\roman*)}},
    leftmargin=3em,
    itemsep=6pt,
    topsep=6pt,
    parsep=0pt
]

\item\label{OptB:range}\qquad
$\begin{aligned}[t]
&\min\{\sigma_0,\sigma_2\}>1,\quad \sigma_1>\sigma_0,\quad
0\leq\eta<1,\quad h\geq0,\ T_0>T>0,\\
&\qquad 2(\sigma_0+h)(\sigma_1-\sigma_0)
+(\sigma_1-\sigma_0)^2\leq T_0^2,
\end{aligned}$

\item\label{OptB:threshold}\qquad
$ \max\big\{0,\,\frac{J_6^*(\bsigma;h,\eta)}{T_0^2}+\frac{J_8^*(\bsigma;h,\eta)}{(\sigma_1-\sigma_0)T_0^4}\big\}\leq J_4(\bsigma;h,\eta)$,

\item\label{OptB:Omega}\qquad
$\Omega_{\bsigma}(\beta,\gamma;h,\eta)\leq  0\qquad
 \text{ for }\quad  1/2\leq \beta\leq  1,\ T\leq \gamma\leq  T_0,$
 
\item\label{OptB:Xi}\qquad
$\Xi_{\bsigma,h,\eta}(u)\geq 0
\quad\text{ for }\quad u\geq 2,$

\item\label{OptB:Delta}\qquad
$\Delta_{\mathbb R}(\bsigma;h,\eta)>0,\quad
\Delta_{\mathbb C}(\bsigma;h,\eta)>0.$\\

\end{enumerate}
\end{problem}

\subsection{Nonnegativity of the prime ideal contribution - Condition \ref{Optimisation:B}\ref{OptB:Xi}}
Recall definition \eqref{E:def}. For given parameters $(\bsigma,h,\eta)$, we want to verify the simpler sufficient (but not necessary) property $\Xi_{\bsigma,h,\eta}(u)\ge 0$, $u\geq 2$.
 This way, we will be able to discard the prime ideal contribution to \eqref{expression}. Writing $u=e^t$, we have
\begin{align*}
\Xi^*_{\bsigma,h,\eta}=\Xi_{\bsigma,h,\eta}(e^t)&:t\geq \log 2\mapsto
 (\sigma_1-\sigma_0)\bigg[\frac{t}{e^{\sigma_0t}-1}-\frac{\eta t}{e^{(\sigma_0+h)t}-1}\bigg]+\log \bigg(\frac{1-e^{-\sigma_0 t}}{1-e^{-\sigma_1 t}}\bigg)\\&
     \qquad-\eta\log\bigg(\frac{1-e^{-(\sigma_0+h)t}}{1-e^{-(\sigma_1+h)t}}\bigg)- \frac{(\sigma_1-\sigma_0)^2}{2}\bigg[
        \frac{t^2 e^{\sigma_2t}}{(e^{\sigma_2 t}-1)^2}-\frac{\eta t^2e^{(\sigma_2+h)t}}{(e^{(\sigma_2+h)t}-1)^2}\bigg].\nonumber
\end{align*}
For a constant $\alpha>0$, by expanding into power series,
\begin{equation*}
\frac{1}{e^{\alpha t}-1}=\sum_{m\ge1}e^{-\alpha mt},\qquad
\log(1-e^{-\alpha t})=-\sum_{m\ge1}\frac{e^{-\alpha mt}}{m},
\qquad\frac{e^{\alpha t}}{(e^{\alpha t}-1)^2}=
 \sum_{m\ge1}m e^{-\alpha mt},
\end{equation*}
we obtain
\begin{equation}\label{Stechkin}
\Xi^*_{\bsigma,h,\eta}(t) =\sum_{m\ge1}\frac{e^{-\sigma_0mt}(1-\eta e^{-hmt})}{m} W_\theta((\sigma_1-\sigma_0)mt),
\end{equation}
where
\begin{equation*}
W_\theta(x) := x-1+e^{-x} -\frac{1}{2}x^2e^{-\theta x}\quad\text{ and }\quad
\theta=\theta(\bsigma):=\frac{\sigma_2-\sigma_0}{\sigma_1-\sigma_0}.
\end{equation*}
Observe that for any $\theta$, $W_\theta(0)=0$. The behaviour of $W_\theta(x)$ depends naturally on $\theta$.
\begin{lemma}\label{calculus} 
Let $\theta>0$.
\begin{enumerate}[
    label=\textnormal{\textbf{(\alph*)}},
    ref=\textnormal{\textbf{(\alph*)}},
    leftmargin=3em,
    itemsep=6pt,
    topsep=6pt,
    parsep=0pt
]
\item\label{>=1/3} If $\theta\ge 1/3$, then $W_\theta(x)>0$ for $x>0$.
\item\label{<1/3} If $0<\theta<1/3$, then $W_\theta$ has exactly one positive
zero $x_\theta$. Furthermore, $W_\theta(x)<0$ for $0<x<x_\theta$, and $W_\theta(x)>0$ for $x>x_\theta$.
\end{enumerate}
\end{lemma}
\begin{proof}
For $x>0$, write $W_\theta(x)= x^2(I(x)-e^{-\theta x})/2$
where
\begin{equation}\label{def:I}
 I(x):=\frac{2(x-1+e^{-x})}{x^2}.
\end{equation}
Since $e^t>1+t$ for all $t\in\mathbb{R}$, $t\neq 0$, we have $I(x)>0$, $x\in\mathbb{R}$, $x\neq 0$. Moreover, $I$ extends continuously to $x=0$ by defining $I(0)=\lim_{x\to0}I(x)=1$. Thus $I(x)>0$ for all  $x\in\mathbb{R}$. We may now define the infinitely differentiable function
\begin{equation*}
        \Theta:x>0\mapsto-\frac{\log I(x)}{x}.
\end{equation*}
Then we have $I(x)=e^{-x\Theta(x)}$ and
$W_\theta(x) =x^2( e^{-x\Theta(x)}-e^{-\theta x})/2$.
Hence $W_\theta(x)>0$ if and only if $\Theta(x)<\theta$
and $W_\theta(x)=0$ if and only if $ \Theta(x)=\theta$.
       
We need to study the monotonicity of $\Theta$. By definition \eqref{def:I},
we have
\begin{align*}
 (-\log I(x))'&=\phantom{-}\frac{x-2+(x+2)e^{-x}}{x(x-1+e^{-x})},\\
(-\log I(x))''&= -\frac{x^2-4x+2+(x^3+x^2+4x-4)e^{-x}+ 2e^{-2x} } {x^2(x-1+e^{-x})^2 }.
\end{align*}
By calculus, we can prove that $x^2-4x+2+(x^3+x^2+4x-4)e^{-x}+ 2e^{-2x}> 0$ for all $x>0$. Thus $(-\log I(x))''<0$ for $x> 0$ and hence $-\log I(x)$ is strictly concave for $x>0$. 
Moreover, since $-\log I(0)=0$, strict concavity gives for all $0<x<y$
\begin{equation*}
-\log I(x)=-\log I\bigg(\bigg(1-\frac{x}{y}\bigg)0+\frac{x}{y}y\bigg)>-\frac{x}{y}\log I(y).
\end{equation*}
Dividing by $x$, we obtain $\Theta(x)>\Theta(y)$ and hence $\Theta(x)$ is strictly decreasing for $x>0$.
Next, since $-\log I(0)=0$ and $(-\log I)'(0)=1/3$, we readily see that $\lim_{x \to 0^+}\Theta(x)= 1/3$.
Furthermore, since $I(x) =2(x-1+e^{-x})/x^2 \sim 2/x$, we have
\begin{equation*}
 \Theta(x)=-\frac{\log I(x)}{x} =\frac{\log x-\log2+o(1)}{x}\to 0,
\end{equation*}
so that $\lim_{x \to \infty}\Theta(x)= 0$ and $\Theta$ maps $(0,\infty)$ strictly decreasingly onto
$(0,1/3)$.

Subsequently,  if $\theta\ge 1/3$, then $\Theta(x)< 1/3\leq  \theta$ for all $x>0$, and hence
$W_\theta(x)>0$ for all $x>0$, whence \ref{>=1/3}. Else, if $0<\theta<1/3$, then there exists
a unique $x_\theta>0$ such that $\Theta(x_\theta)=\theta$. Since $\Theta$ is strictly decreasing, $x_\theta$ is the
unique positive zero of $W_\theta$. Thereupon,
we have $\Theta(x)>\theta$ if $0<x<x_\theta$, and
$ \Theta(x)<\theta$ if $x>x_\theta$; consequently
$W_\theta(x)<0$ if $0<x<x_\theta$, and
$ W_\theta(x)>0$ if $x>x_\theta$. This gives \ref{<1/3}.
\end{proof}

For a given $\theta=\frac{\sigma_2-\sigma_0}{\sigma_1-\sigma_0}$, by Lemma \ref{calculus}\ref{>=1/3}, if $\theta\geq1/3$, then $W_\theta(x)\geq0$ for every $x\geq0$.
Since $1-\eta e^{-hmt}>0$,
every term in \eqref{Stechkin} is non-negative. Hence
\begin{equation}\label{positive:1}
\Xi^*_{\bsigma,h,\eta}(t)\geq 0,\qquad\text{ for }t\geq\log 2.
\end{equation}
Suppose now $0<\theta<1/3$. Let
$0<x_\theta^{\mathrm l}<x_\theta^{\mathrm u}$ be such that $x_\theta\in(x_\theta^{\mathrm l},x_\theta^{\mathrm u})$.
By Lemma \ref{calculus}\ref{<1/3}, $W_\theta(x)\geq0$ for $x\geq x_\theta$. Therefore, if
$t\geq \frac{x_\theta^{\mathrm u}}{\sigma_1-\sigma_0}$,
then $(\sigma_1-\sigma_0)mt\geq (\sigma_1-\sigma_0)t\geq x_\theta$
and every term in \eqref{Stechkin} is non-negative. Hence, we also have
\begin{equation}\label{positive:2}
 \Xi^*_{\bsigma,h,\eta}(t)\geq 0,\qquad t\geq\frac{x_\theta^{\mathrm u}}{\sigma_1-\sigma_0}.
\end{equation}
Thus \eqref{positive:1} and \eqref{positive:2} present ranges of $t$, and hence of $u=e^t$, in which we have $\Xi_{\bsigma,h,\eta}(u)\ge 0$. It remains to treat the range
\begin{equation*}
 \log2\leq  t\leq 
 \frac{x_\theta^{\mathrm u}}{\sigma_1-\sigma_0}.
\end{equation*}
in the case $0<\theta<1/3$. Let
\begin{equation*}
v=\left\lceil\frac{x_\theta^{\mathrm u}}{(\sigma_1-\sigma_0)\log2}\right\rceil.
\end{equation*}
Then $v(\sigma_1-\sigma_0)t\geq v(\sigma_1-\sigma_0)\log2 >x_\theta$ and thanks to Lemma \ref{calculus}\ref{<1/3}, all terms with $m\geq v$ in the sum defining \eqref{Stechkin} are non-negative. Hence, for any threshold $\mathcal  K\in\mathbb N$,
\begin{equation}\label{S_Kv:}
\Xi^*_{\bsigma,h,\eta}(t)\geq
  S_{\mathcal Kv}(t):=\sum_{m=1}^{\mathcal Kv}\frac{e^{-\sigma_0mt}(1-\eta e^{-hmt})}{m} W_\theta((\sigma_1-\sigma_0)mt),
\end{equation}
and our goal is to be able to prove that 
\begin{equation}\label{S_Kv}
 S_{\mathcal Kv}(t)\geq 0\qquad\text{ for }\quad \log 2\leq  t\leq  \frac{x_\theta^{\text{u}}}{\sigma_1-\sigma_0}
\end{equation} 
for a given $\mathcal  K\in\mathbb{N}$. The number $\mathcal  K$ is chosen for simplicity to be as small as possible, to ensure that at least $(\mathcal  K-1)v+1$ summands (the ones with $m\geq v$) contribute positively to $ S_{\mathcal Kv}(t)$, increasing the chances of accomplishing task \eqref{S_Kv}. 

We arrive at two optimisation problems. Both imply Problem \ref{Optimisation:B} and thus Problem \ref{Optimisation:A}.

\begin{problem}\label{Optimisation:C}
 
Find parameters $(\bsigma,h,\eta,T,T_0)$ such that 
\begin{enumerate}[
    label=\textnormal{\textbf{(\roman*)}},
    ref=\textnormal{\textbf{(\roman*)}},
    leftmargin=3em,
    itemsep=6pt,
    topsep=6pt,
    parsep=0pt
]

\item\label{OptC:range}\qquad
$\begin{aligned}[t]
&\min\{\sigma_0,\sigma_2\}>1,\quad \sigma_1>\sigma_0,\quad
0\leq\eta<1,\quad h\geq0,\ T_0>T>0,\\
&\qquad 2(\sigma_0+h)(\sigma_1-\sigma_0)
+(\sigma_1-\sigma_0)^2\leq T_0^2,
\end{aligned}$

\item\label{OptC:threshold}\qquad
$ \max\big\{0,\,\frac{J_6^*(\bsigma;h,\eta)}{T_0^2}+\frac{J_8^*(\bsigma;h,\eta)}{(\sigma_1-\sigma_0)T_0^4}\big\}\leq J_4(\bsigma;h,\eta)$,

\item\label{OptC:Omega}\qquad
$\Omega_{\bsigma}(\beta,\gamma;h,\eta)\leq  0\qquad
 \text{ for }\quad   1/2\leq \beta\leq  1,\ T\leq \gamma\leq  T_0,$
 
\item\label{OptC:Theta}\qquad
$\theta(\bsigma)=\frac{\sigma_2-\sigma_0}{\sigma_1-\sigma_0}\geq 1/3,$

\item\label{OptC:Delta}\qquad
$\Delta_{\mathbb R}(\bsigma;h,\eta)>0,\quad
\Delta_{\mathbb C}(\bsigma;h,\eta)>0.$\\

\end{enumerate}
\end{problem}

\begin{problem}\label{Optimisation:D}
 
Find parameters $(\bsigma,h,\eta,T,T_0,\mathcal K)$ such that 
\begin{enumerate}[
    label=\textnormal{\textbf{(\roman*)}},
    ref=\textnormal{\textbf{(\roman*)}},
    leftmargin=3em,
    itemsep=6pt,
    topsep=6pt,
    parsep=0pt
]

\item\label{OptD:range}\qquad
$\begin{aligned}[t]
&\min\{\sigma_0,\sigma_2\}>1,\quad \sigma_1>\sigma_0,\quad
0\leq\eta<1,\quad h\geq0,\ T_0>T>0,\\
&\qquad 2(\sigma_0+h)(\sigma_1-\sigma_0)
+(\sigma_1-\sigma_0)^2\leq T_0^2,\quad \mathcal  K\in\mathbb{N},
\end{aligned}$

\item\label{OptD:threshold}\qquad
$ \max\big\{0,\,\frac{J_6^*(\bsigma;h,\eta)}{T_0^2}+\frac{J_8^*(\bsigma;h,\eta)}{(\sigma_1-\sigma_0)T_0^4}\big\}\leq J_4(\bsigma;h,\eta)$,

\item\label{OptD:Omega}\qquad
$\Omega_{\bsigma}(\beta,\gamma;h,\eta)\leq  0\qquad
 \text{ for }\quad 1/2\leq \beta\leq  1,\ T\leq \gamma\leq  T_0,$
 
\item\label{OptD:Theta}\qquad
$0<\theta(\bsigma)=\frac{\sigma_2-\sigma_0}{\sigma_1-\sigma_0}<1/3,$

\item\label{OptD:S1}\qquad
$ S_{\mathcal Kv}(t)\geq 0$\quad for\quad $\log 2\le t\leq x_\theta^{\text{u}}/(\sigma_1-\sigma_0)$,

\item\label{OptD:Delta}\qquad
$\Delta_{\mathbb R}(\bsigma;h,\eta)>0,\quad
\Delta_{\mathbb C}(\bsigma;h,\eta)>0.$

\end{enumerate}
\end{problem}

\section{Outputs}\label{outputs}

\subsection{Recovering Hoffstein's result}

Hoffstein solves Problem \ref{Optimisation:D} with $T=0.87$. More precisely, Hoffstein's parameters clearly satisfy Condition \ref{Optimisation:D}\ref{OptD:range} and are defined as
\begin{equation}\label{Hoffstein:parameters}
 (\bsigma,h,\eta,T,T_0,\mathcal K)=((1.01,1.2,1.071),0,0,0.87,16,4).
\end{equation}
Note that when $\eta=0$, the value of $h$ is in fact irrelevant. Condition \ref{Optimisation:D}\ref{OptD:Theta} holds since
\begin{equation*}
\theta(\bsigma)=\frac{\sigma_2-\sigma_0}{\sigma_1-\sigma_0}= \frac{1.071-1.01}{1.2-1.01}=\frac{61}{190}=0.321052\ldots<\frac{1}{3}.
\end{equation*}
Using rigorous interval arithmetic, we find $W_\theta(0.4514)=5.54701\ldots\cdot10^{-9}>0$.
It follows from the sign properties of $W_\theta$ that $x_\theta<0.4514$ and we may therefore take $x_\theta^{\mathrm u}=0.4514$.
Moreover, observe that
\begin{equation*}
J_4(\bsigma;0,0)=0.02805\ldots>0,\quad J_6^*(\bsigma;0,0)=6.73550\ldots,\quad
 J_8^*(\bsigma;0,0)=8.51817\ldots.
\end{equation*}
Taking $T_0=16$, we obtain
\begin{equation*}
        \max\bigg\{0,\,\frac{J_6^*(\bsigma;0,0)}{T_0^2}+\frac{J_8^*(\bsigma;0,0)}{(\sigma_1-\sigma_0)T_0^4}\bigg\}=
 0.02699\ldots<
   J_4(\bsigma;0,0) =0.02805\ldots.
\end{equation*}
Hence condition \ref{Optimisation:D}\ref{OptD:threshold} holds. By \eqref{condition1}, this implies
\begin{equation*}
   \Omega_{\bsigma}(\beta,\gamma;0,0)\leq 0 \qquad\text{ for }\quad 1/2\leq \beta\leq  1,\ \gamma\geq 16.
\end{equation*}
On the other hand, we have
\begin{equation*}
   v =\left\lceil\frac{x_\theta^{\mathrm u}}
      {(\sigma_1-\sigma_0)\log2}\right\rceil=\left\lceil \frac{0.4514}
{(1.2-1.01)\log2}\right\rceil=4.
\end{equation*}
Consequently, since $\mathcal K=4$, according to definition \eqref{S_Kv:} the finite sum is 
\begin{equation*}
  S_{16}(t) =\sum_{m=1}^{16}\frac{e^{-\sigma_0mt}}{m}W_\theta((\sigma_1-\sigma_0)mt).
\end{equation*}
Here the factor $1-\eta e^{-hmt}$ is equal to $1$, since
$\eta=0$. We next verify conditions \ref{Optimisation:D}\ref{OptD:Omega} and \ref{Optimisation:D}\ref{OptD:S1}, namely
\begin{align*}
\Omega_{\bsigma}(\beta,\gamma;0,0)
&\leq  0\qquad\text{ for }\quad 1/2\leq \beta\leq  1,\quad 0.87\leq \gamma\leq  16,\\
 S_{16}(t)&\geq 0,\qquad\text{ for }\quad \log 2\le t\le\frac{x_\theta^{\text{u}}}{\sigma_1-\sigma_0}=\frac{0.4514}{1.2-1.01}= 2.37578\ldots.
\end{align*}
Both conditions involve a compact verification and this is checked by interval arithmetic. 

Finally, rigorous interval arithmetic gives
\begin{equation*}
\Delta_{\mathbb R}(\bsigma;0,0)
=[8.7812\cdot10^{-6}\pm 4.36\cdot10^{-11}],\quad
\Delta_{\mathbb C}(\bsigma;0,0)
=[3.589\cdot10^{-6}\pm 4.84\cdot10^{-10}].
\end{equation*}
In particular, $\Delta_{\mathbb R}(\bsigma;0,0)>0$, $\Delta_{\mathbb C}(\bsigma;0,0)>0$,
which verifies condition \ref{Optimisation:D}\ref{OptD:Delta}.

Therefore, as in \S\ref{Argument:1}, we have $h(K)<0.87$ for sufficiently large $n_K$. 

%%%%%%%
\subsection{A direct improvement on Hoffstein's result}
With the same choice of parameters as in \eqref{Hoffstein:parameters}, we can actually decrease the threshold from $0.87$ to $0.825$. Thus the choice
\begin{equation*}
(\bsigma,h,\eta,T,T_0,\mathcal K)=((1.01,1.2,1.071),0,0,0.825,16,4)
\end{equation*}
also solves Problem \ref{Optimisation:D}, and we have 
$h(K)<0.825$  for sufficiently large $n_K$.
 %%%%%%%
 
\subsection{An improvement via Stechkin's trick and a proof of Theorem \ref{main}.}\label{proof-main}

We now allow the parameters $h$ and $\eta$ to be non-zero and use Stechkin's trick \eqref{shifted}. We take
\begin{equation}\label{Stechkin:parameters}
(\bsigma,h,\eta,T,T_0,\mathcal K)=((1.00001,1.00614,1.00205),0.00070,0.99769,0.58,40000,4),
\end{equation}
which satisfies Condition \ref{Optimisation:D}\ref{OptD:range}. Moreover, we have
\begin{equation*}
\theta(\bsigma)=\frac{\sigma_2-\sigma_0}{\sigma_1-\sigma_0}
=\frac{204}{613}=0.33278\ldots<\frac{1}{3},
\end{equation*}
and hence condition \ref{Optimisation:D}\ref{OptD:Theta} holds.

Next, using rigorous interval arithmetic, we find $
W_\theta(0.01960)=7.32363\ldots\cdot10^{-13}>0$.
It follows from Lemma \ref{calculus}\ref{<1/3} that $x_\theta<0.01960$, and that we may take $x_\theta^{\mathrm u}=0.01960$. Moreover,
\begin{equation*}
v=\left\lceil\frac{x_\theta^{\mathrm u}}{(\sigma_1-\sigma_0)\log2}\right\rceil
=\left\lceil\frac{0.01960}{0.00613\log2}\right\rceil=5.
\end{equation*}
Since $\mathcal K=4$, the finite sum in definition \eqref{S_Kv:} is therefore $S_{20}(t)$.
We then verify condition \ref{Optimisation:D}\ref{OptD:S1} via rigorous interval arithmetic by observing that
\begin{equation*}
S_{20}(t)>0\qquad\text{ for }\quad
\log2\leq t\leq\frac{0.01960}{0.00613}=3.19738\ldots.
\end{equation*}
With respect to Condition \ref{Optimisation:D}\ref{OptD:threshold}, the quantities concerned are
\begin{equation*}
J_4(\bsigma;h,\eta)=4.73884\ldots\cdot 10^{-8},\quad J_6^*(\bsigma;h,\eta)=10.08480\ldots,\quad J_8^*(\bsigma;h,\eta)=8.25034\ldots.
\end{equation*}
Taking $T_0=40000$, we obtain
\begin{equation*}
\max\bigg\{0,\frac{J_6^*(\bsigma;h,\eta)}{T_0^2}
+\frac{J_8^*(\bsigma;h,\eta)}{(\sigma_1-\sigma_0)T_0^4}\bigg\}=6.30300\ldots\cdot10^{-9}<J_4(\bsigma;h,\eta).
\end{equation*}
Consequently, \eqref{condition1} gives
$\Omega_{\bsigma}(\beta,\gamma;h,\eta)\leq 0$ for $1/2\leq\beta\leq 1$, $\gamma\geq 40000$.
On the remaining box we have
\begin{equation*}
\Omega_{\bsigma}(\beta,\gamma;h,\eta)<0\qquad\text{ for }\quad
1/2\leq\beta\leq 1,\quad0.58\leq\gamma\leq40000.
\end{equation*}
This verifies condition \ref{Optimisation:D}\ref{OptD:Omega}. Finally, interval arithmetic gives
\begin{equation}\label{deltas}
\Delta_{\mathbb R}(\bsigma;h,\eta)=\big[3.2818\cdot10^{-14}\pm 2\cdot10^{-19}\big],\quad\Delta_{\mathbb C}(\bsigma;h,\eta)=\big[6.7903\cdot10^{-16}\pm 3\cdot10^{-21}\big].
\end{equation}
Both intervals are strictly positive, and hence condition \ref{Optimisation:D}\ref{OptD:Delta} also holds.

Therefore the parameters \eqref{Stechkin:parameters} solve Problem \ref{Optimisation:D} with $T=0.58$, and we prove Theorem \ref{main}.

\subsection{Proof of Theorem \ref{main:quantitative}}\label{proof:main:quantitative}
We consider the same parameters \eqref{Stechkin:parameters} used
to prove Theorem \ref{main}. These parameters solve Problem
\ref{Optimisation:D}, and hence also Problem \ref{Optimisation:A}. In particular, we have
$\Omega_{\bsigma}(\beta,\gamma;h,\eta)\leq 0$, for $1/2\leq\beta\leq 1$ and $\gamma\geq 0.58$,
and $\Xi_{\bsigma,h,\eta}(u)\geq0$ for every $u\geq 2$. Furthermore, estimates \eqref{deltas} give  
$\Delta_{\mathbb R}(\bsigma;h,\eta)>3.2817\cdot10^{-14}$ and $\Delta_{\mathbb C}(\bsigma;h,\eta)>6.7902\cdot10^{-16}$.
Using rigorous interval arithmetic, we also obtain
\begin{equation*}
\Delta_O(\bsigma;h,\eta)
=-598.23409\ldots>-599.
\end{equation*}
Recall \eqref{expression}.
Since $n_K=r_1+2r_2$, the non-negativity of the prime ideal contribution gives
\begin{align*}
&(\sigma_1-\sigma_0)\mathcal{F}_1(\sigma_0;h,\eta)
+\mathcal{F}_2(\sigma_0,\sigma_1;h,\eta)
+\frac{1}{2}(\sigma_1-\sigma_0)^2
\mathcal{F}_3(\sigma_2;h,\eta)\\
&\qquad\qquad\geq
r_1\Delta_{\mathbb R}(\bsigma;h,\eta)
+r_2\Delta_{\mathbb C}(\bsigma;h,\eta)
+\Delta_O(\bsigma;h,\eta)\geq
3.395\cdot10^{-16}n_K-599.
\end{align*}

On the other hand, by the  argument of \S\ref{quan}, using rigorous interval arithmetic, we calculate 
\begin{equation*} 
\max\bigg\{0,\max_{(\beta,\gamma)\in[1/2,1]\times[0,0.58]}\Omega_{\bsigma}(\beta,\gamma;h,\eta)\bigg\}<\frac{96}{5}.
\end{equation*}
Thus we may take $M_{0.58}(\bsigma;h,\eta)=96/5$ such that
\begin{flalign*}
&& N_K(0.58)\geq \frac{2}{M_{0.58}}(3.395\cdot10^{-16}n_K-599)=
\frac{10}{96}(3.395\cdot10^{-16}n_K-599)
\gg n_K. && \text{\qedsymbol}
\end{flalign*}

%%%%%%%%%%%%%%%%%%%%%%%%%%%%%%%%%%%%%
\section{Concluding remarks}\label{numerics}
%%%%%%%%%%%%%%%%%%%%%%%%%%%%%%%%%%

\subsection{Numerical remarks} The verifications in \S\ref{outputs} were performed in SageMath 10.9 \cite{SageMath*26}. We used Sage's built-in interval arithmetic, together with its built-in implementations of $\log$, $\exp$, $\log\Gamma$, $\psi$ and its derivatives. All three applications in \S\ref{outputs} correspond to Problem \ref{Optimisation:D}.

The inequality for $S_{Kv}$, namely condition \ref{Optimisation:D}\ref{OptD:S1}, was checked by standard interval subdivision algorithms and the mean value theorem for an interval. The tail condition \ref{Optimisation:D}\ref{OptD:threshold} as well as the quantities $\Delta_{\mathbb R}$ and $\Delta_{\mathbb C}$ in \ref{Optimisation:D}\ref{OptD:Delta} were evaluated using rigorous interval arithmetic. 

The most challenging part was verifying condition \ref{Optimisation:D}\ref{OptD:Omega}, which was also checked by a standard box subdivision algorithm via the mean value theorem in \emph{two variables} for a box. Indeed, for a numerical experiment, we fix $\bsigma$, $h$ and $\eta$ so that $\Omega_{\bsigma}(\beta,\gamma;h,\eta)$ becomes a function in $\beta$ and $\gamma$. For each box $\mathcal B$ in the subdivision, we used the following inequality
\begin{equation*}
\sup_{(x,g)\in \mathcal B}\widetilde{\Omega}(x,g)
\leq
\widetilde{\Omega}(x_c,g_c)
+\frac{x_+-x_-}{2}\sup_{(x,g)\in \mathcal B}
\bigg|\frac{\partial\widetilde{\Omega}}{\partial x}(x,g)\bigg|
+\frac{g_+-g_-}{2}\sup_{(x,g)\in \mathcal B}
\bigg|\frac{\partial\widetilde{\Omega}}{\partial g}(x,g)\bigg|,
\end{equation*}
where $x=(\beta-1/2)^2$, $g=\gamma^2$, $\widetilde{\Omega}(x,g)=\Omega_{\bsigma}(1/2+\sqrt{x},\sqrt{g};h,\eta)$, and $\mathcal B=[x_-,x_+]\times[g_-,g_+]$ is a box with centre $(x_c,g_c)$. The values of $x$ and $g$ are not random: they come from exploiting the symmetry of the weight $\Omega_{\bsigma}$ and the many algebraic simplifications that arise. The subdivision thus covered the box $[0,1/4]\times[T^2,T_0^2]$ and the partial derivatives were bounded using rigorous interval arithmetic. A box was accepted by the algorithm only whenever the resulting upper bound was strictly negative; otherwise, it was subdivided further. 

Unfortunately, we were unable to improve upon the threshold $T=0.58$ of Theorem \ref{main}. The complete SageMath program is included in the arXiv ancillary files.

\subsection{Future work} We note that Stechkin's trick, which we introduced in \S \ref{peck}, is optimal in a certain sense --- see for instance the remark in \cite{Stechkin*70}. However, recent work by Bellotti, Yang, and the first author \cite{BellottiTrudgianYang*26} shows that one may combine this with other ideas to obtain improvements in related problems. We have not explored this, but, given the utility of Stechkin's trick to this problem, we recommend further research on this point.

In \S \ref{horse}, we noted that Odlyzko's survey \cite{Odlyzko*90} contained three results relevant to our work. We looked at the third in our Theorem \ref{main}. The methods here are effective, and should be able to recover the second claim, namely, that $\zeta_K$ has a zero below $14$ for every $K$ with $n_K\geq 2$: this already follows from Theorem \ref{main} for sufficiently large $n_K$. Nonetheless, for the remaining bounded degrees, additional arguments would be required, possibly supplemented by computations using the LMFDB \cite{LMFDB}. We leave this as a task for future research.

\Addresses

\end{document}